\documentclass[11pt]{amsart}

\usepackage{amssymb,amsmath,amsthm,newlfont}
\usepackage[shortlabels]{enumitem}
\usepackage{hyperref}

\theoremstyle{plain}
\newtheorem{theorem}{Theorem}[section]
\newtheorem{proposition}[theorem]{Proposition}
\newtheorem{corollary}[theorem]{Corollary}
\newtheorem{lemma}[theorem]{Lemma}
\newtheorem{question}[theorem]{Question}

\theoremstyle{definition}
\newtheorem{definition}[theorem]{Definition}

\theoremstyle{remark}
\newtheorem*{remark}{Remark}

\numberwithin{equation}{section}

\renewcommand{\hat}{\widehat}
\renewcommand{\tilde}{\widetilde}

\newcommand{\CC}{\mathbb{C}}
\newcommand{\DD}{\mathbb{D}}
\newcommand{\RR}{\mathbb{R}}

\newcommand{\Coo}{\hat{\CC}}
\newcommand{\cD}{\mathcal{D}}
\newcommand{\cH}{\mathcal{H}}
\newcommand{\cL}{\mathcal{L}}
\newcommand{\cP}{\mathcal{P}}
\newcommand{\cQ}{\mathcal{Q}}
\newcommand{\cU}{\mathcal{U}}
\newcommand{\cV}{\mathcal{V}}
\renewcommand{\Re}{\operatorname{Re}}
\renewcommand{\Im}{\operatorname{Im}}
\DeclareMathOperator{\supp}{supp}
\DeclareMathOperator{\diam}{diam}
\DeclareMathOperator{\loc}{loc}

\begin{document}

\date{6 April 2023}

\title{Holomorphic motions, dimension, area and quasiconformal mappings}

\author[A. Fuhrer]{Aidan Fuhrer}
\address{D\'epartement de math\'ematiques et de statistique, Universit\'e Laval,
Qu\'ebec City (Qu\'ebec),  Canada G1V 0A6.}
\email{aidan.fuhrer.1@ulaval.ca}

\author[T. Ransford]{Thomas Ransford}
\address{D\'epartement de math\'ematiques et de statistique, Universit\'e Laval,
Qu\'ebec City (Qu\'ebec),  Canada G1V 0A6.}
\email{ransford@mat.ulaval.ca}

\author[M. Younsi]{Malik Younsi}
\address{Department of Mathematics, University of Hawaii Manoa, Honolulu, HI 96822, USA.}
\email{malik.younsi@gmail.com}

\thanks{Fuhrer supported by an NSERC Canada Graduate Scholarship.
Ransford supported by grants from NSERC and the Canada Research Chairs program.
Younsi supported by NSF Grant DMS-2050113.}

\begin{abstract}
We describe the variation of the Minkowski, packing
and Hausdorff dimensions of a
set moving under a holomorphic motion,
as well as the variation of its area.
Our method provides a new, unified approach to various celebrated theorems about quasiconformal mappings, including the work of Astala on the distortion of area and dimension under quasiconformal mappings and the work of Smirnov on the dimension of quasicircles.
\end{abstract}


\makeatletter
\@namedef{subjclassname@2020}{\textup{2020} Mathematics Subject Classification}
\makeatother


\maketitle

\section{Introduction}\label{S:intro}

In what follows,
we write $\DD,\CC$ and $\Coo$ for the open unit disk,
the complex plane and the Riemann sphere, respectively.

\begin{definition}\label{D:holomotion}
Let $A$ be a subset of  $\Coo$.
A \textit{holomorphic motion} of $A$ is a map
$f:\DD\times A \to \Coo$ such that:
\begin{enumerate}[(i)]
\item for each fixed $z\in A$,
the map $\lambda\mapsto f(\lambda,z)$
is holomorphic on $\DD$;
\item for each fixed $\lambda \in\DD$,
the map $z\mapsto f(\lambda,z)$ is
injective on $A$;
\item for all $z\in A$, we have $f(0,z)=z$.
\end{enumerate}
We write $f_\lambda(z):=f(\lambda,z)$
and $A_\lambda:=f_\lambda(A)$.
We sometimes abuse terminology
and call the set-valued map $\lambda\mapsto A_\lambda$
a holomorphic motion of~$A$.
\end{definition}

Holomorphic motions were introduced by Ma\~{n}\'{e}, Sad and Sullivan \cite{MSS83}.
They established the so-called $\lambda$-lemma,
which says that every holomorphic motion $f:\DD \times A \to \Coo$
has an extension to a holomorphic motion $F:\DD \times \overline{A} \to \Coo$,
and that $F$ is jointly continuous in $(\lambda,z)$.
They  exploited this result to  describe the variation of
Julia sets of holomorphic families of hyperbolic rational maps.
Holomorphic motions have since been applied in
various other areas of dynamical systems, notably
in describing the variation of limit sets of Kleinian groups, see e.g.\ \cite[\S12.2.1]{AIM09}.

Consider the following  problem.
Let $\lambda\mapsto A_\lambda$ be a holomorphic motion such that
$A_\lambda\subset\CC$ for all $\lambda\in\DD$.
What sort of functions are $\lambda\mapsto\dim(A_\lambda)$ and $\lambda\mapsto|A_\lambda|$?
Here $|\cdot|$ denotes the area measure (two-dimensional Lebesgue measure) and $\dim(\cdot)$ can denote any reasonable notion of dimension.
Various aspects of this problem have been treated
in the literature, see for example
\cite{As94,AZ95,BR06,CT23,EH95,GV73,Ka00,Pr07,PS11,Ra93,Ru82,Sm10}.
We shall discuss some of these contributions in more detail later.

In this article, we shall be mainly interested in three notions of dimension, namely
the Minkowski, packing and Hausdorff dimensions. To  state our results, it is convenient to introduce another definition.

\begin{definition}\label{D:infharmonic}
Let $D$ be a domain in $\CC$. A positive function $u:D\to[0,\infty)$ is called
\emph{inf-harmonic} if there exists a  family $\cH$
of harmonic functions on $D$ such that $u(\lambda)=\inf_{h\in\cH}h(\lambda)$
for all $\lambda\in D$.
\end{definition}

In Theorems~\ref{T:Minkowski}--\ref{T:Hausdorff},
we consider a holomorphic motion $f:\DD\times A\to\CC$
of a subset $A$ of $\CC$, and write $A_\lambda:=f_\lambda(A)$.

Our first result describes the variation of the Minkowski dimension,
or more precisely the upper Minkowski dimension $\overline{\dim}_M$,
of a bounded set moving under a holomorphic motion.

\begin{theorem}\label{T:Minkowski}
Let $\lambda\mapsto A_\lambda$  be a holomorphic motion of a bounded subset $A$ of $\CC$.
Then $A_\lambda$ is bounded for all $\lambda\in\DD$, and
either $\overline{\dim}_M(A_\lambda)=0$ for all $\lambda\in\DD$,
or $\lambda\mapsto 1/ \overline{\dim}_M(A_\lambda)$ is an inf-harmonic function on $\DD$.
\end{theorem}

From this theorem, we deduce an analogous result for the
packing dimension $\dim_P$.

\begin{theorem}\label{T:packing}
Let $\lambda\mapsto A_\lambda$  be a holomorphic motion of a subset $A$ of $\CC$.
Then either $\dim_P(A_\lambda)=0$ for all $\lambda\in\DD$,
or $\lambda\mapsto 1/ \dim_P(A_\lambda)$ is an inf-harmonic function on $\DD$.
\end{theorem}

From these theorems, we obtain the following corollary.

\begin{corollary}\label{C:subharmonic}
Under the respective assumptions of Theorems~\ref{T:Minkowski} and \ref{T:packing},
$\overline{\dim}_M(A_\lambda)$ and $\dim_P(A_\lambda)$
are continuous, logarithmically  subharmonic functions of $\lambda\in\DD$
(and hence also subharmonic on $\DD$).
In particular, if either these functions
attains a maximum on $\DD$, then it is constant.
\end{corollary}

\begin{proof}
As we shall see, an inf-harmonic function is  a
continuous superharmonic function.
Using Jensen's inequality, it is easy to see that,
if $1/v$ is a positive superharmonic function,
then $\log v$ is a subharmonic function, and hence also $v$.
The last part of the corollary is a consequence of the maximum
principle for subharmonic functions.
\end{proof}

For the Hausdorff dimension $\dim_H$, there is a result
similar to Theorems~\ref{T:Minkowski} and \ref{T:packing},
but with a  weaker conclusion.

\begin{theorem}\label{T:Hausdorff}
Let $\lambda\mapsto A_\lambda$  be a holomorphic motion of a subset $A$ of $\CC$.
Then either $\dim_H(A_\lambda)=0$ for all $\lambda\in\DD$,
or $\dim_H(A_\lambda)>0$ for all $\lambda\in\DD$.
In the latter case, $\lambda\mapsto(1/ \dim_H(A_\lambda)-1/2)$
is the supremum of a family of inf-harmonic functions on $\DD$.
\end{theorem}

The nature of the conclusion in Theorem~\ref{T:Hausdorff} does not permit us
to deduce that
$\log \dim_H(A_\lambda)$ or $\dim_H(A_\lambda)$
is a subharmonic function of $\lambda\in\DD$.
We shall return to this problem at the end of the article.

Our next theorem is a sort of converse result.

\begin{theorem}\label{T:converse}
Let $d:\DD\to(0,2]$ be a function such that  $1/d$ is inf-harmonic on $\DD$.
Then there exists a  holomorphic motion $f:\DD\times A\to\CC$
of a compact subset $A$ of~$\CC$ such that,
setting $A_\lambda:=f_\lambda(A)$, we have
$\dim_P(A_\lambda)=\dim_H(A_\lambda)=d(\lambda)$ for all $\lambda\in\DD$.
\end{theorem}

We remark that Theorems~\ref{T:packing} and \ref{T:converse}
together yield a complete characterization of the variation of the packing dimension of a set moving under a holomorphic motion.

The holomorphic motions
that arise from Julia sets of
holomorphic families of hyperbolic rational maps
(as considered in \cite{MSS83}) have the additional
property that their Hausdorff and packing dimensions
vary as  real-analytic
functions of~$\lambda$.
This is a special case of a result of Ruelle \cite{Ru82}.
(Ruelle stated his theorem for Hausdorff dimension,
but it coincides with packing dimension in this case.)
For general holomorphic motions, it is known that the
Hausdorff and packing dimensions need not be real-analytic
(see e.g.\ \cite{AZ95}). The following corollary
of Theorem~\ref{T:converse}
shows that in fact they may have the same lack of smoothness
as an arbitrary concave function.

\begin{corollary}\label{C:concave}
Given a concave function $\psi:\DD\to[0,\infty)$,
there exists a holomorphic motion $f:\DD\times A\to\CC$
of a compact subset $A$ of $\CC$
such that, setting $A_\lambda:=f_\lambda(A)$, we have
\[
\dim_H(A_\lambda)=\dim_P(A_\lambda)=\frac{2}{1+\psi(\lambda)}
\quad(\lambda\in\DD).
\]
\end{corollary}

\begin{proof}
Every positive concave function on $\DD$ is inf-harmonic,
since it is the lower envelope of a family of affine
functions $\lambda\mapsto a\Re(\lambda)+b\Im(\lambda)+c$,
each of which is harmonic on $D$.
Thus the map $\lambda\mapsto\frac{1}{2}(1+\psi(\lambda))$ is
inf-harmonic on $\DD$.
Also, it is clearly bounded below by $1/2$,
so its reciprocal takes values in $(0,2]$.
The result  therefore follows from
Theorem~\ref{T:converse}.
\end{proof}

We now turn to the discussion of the variation of the area of a set $A \subset \CC$ moving under a holomorphic motion $f:\DD\times \CC\to\CC$. As before, for $\lambda \in \DD$, we write $f_\lambda(z):=f(\lambda,z)$ and $A_\lambda:=f_\lambda(A)$. Then each $f_\lambda:\CC \to \CC$ is quasiconformal and we denote its complex dilatation by $\mu_{f_\lambda}$ (see \S\ref{S:holomotions} for the definitions). Our next result gives a partial description of the function $\lambda \mapsto |A_\lambda|$, where $|\cdot|$ denotes area measure.

\begin{theorem}\label{T:Area}
Suppose that there exists a compact subset $\Delta$ of $\CC$  such that, for each $\lambda\in\DD$,
the map $f_\lambda$ is conformal on $\CC\setminus\Delta$
and $f_\lambda(z)=z+O(1)$ near $\infty$.
Let $A$ be a Borel subset of $\Delta$ such that $|A|>0$.
\begin{enumerate}[\normalfont(i)]
\item If $\mu_{f_\lambda}=0$  a.e.\ on $A$,
then $\lambda\mapsto\log(\pi c(\Delta)^2/|A_\lambda|)$ is an inf-harmonic function on $\DD$,
where $c(\Delta)$ denotes the logarithmic capacity of $\Delta$.
\item If $\mu_{f_\lambda}=0$  a.e.\ on $\CC\setminus A$, then $\lambda\mapsto |A_\lambda|$ is an inf-harmonic function on $\DD$.
\end{enumerate}
\end{theorem}
Note that if $|A|=0$, then $|A_\lambda|=0$ for all $\lambda \in \mathbb{D}$, because quasiconformal mappings preserve zero area.

Our approach based on inf-harmonic functions
also permits us to present a unified treatment of several celebrated theorems about the distortion of area and dimension under quasiconformal maps.

We emphasize here that prior works on the distortion of dimension under quasiconformal mappings relied on some of their more involved analytic properties, such as higher order integrability of the Jacobian. Our approach, on the other hand, only requires the fact that quasiconformal mappings satisfy a ``weak'' quasisymmetry property, as stated in Corollary \ref{C:quasisym}.

For instance, a simple application of the Harnack inequality allows us to obtain the following two results. In Theorem \ref{T:qcdistortion}, $\dim$ denotes any one of $\dim_P,\dim_H$ or $\overline{\dim}_M$. (In the case of $\overline{\dim}_M$, we also suppose that $A$ is bounded.)

\begin{theorem}\label{T:qcdistortion}
Let $F:\CC\to\CC$ be a $k$-quasiconformal homeomorphism,
and let $A$ be a subset of $\CC$ such that $\dim(A)>0$. Then
\[
\frac{1}{K}\Bigl(\frac{1}{\dim A}-\frac{1}{2}\Bigr)
\le \Bigl(\frac{1}{\dim F(A)}-\frac{1}{2}\Bigr)\le
K\Bigl(\frac{1}{\dim A}-\frac{1}{2}\Bigr),
\]
where $K:=(1+k)/(1-k)$.
\end{theorem}

For the Hausdorff dimension, the above estimate was first suggested by Gehring and V\"{a}is\"{a}l\"{a} \cite{GV73} and finally proved by Astala \cite[Theorem 1.4]{As94}. For packing dimension it is a special case of a result of Kaufmann \cite[Theorem~4]{Ka00}.

\begin{theorem}\label{T:areadistortion}
Let $F:\CC\to\CC$ be a $k$-quasiconformal homeomorphism which is conformal on $\CC\setminus\Delta$,
where $\Delta$ is a compact set of logarithmic capacity at most~$1$, and such that $F(z)=z+o(1)$ near $\infty$.
Let $A$ be a Borel subset of $\Delta$.
\begin{enumerate}[\normalfont(i)]
\item If $\mu_{F}=0$  a.e.\ on $A$,
then
\begin{equation*}
|F(A)|\le \pi^{1-1/K}|A|^{1/K}.
\end{equation*}
\item If $\mu_{F}=0$  a.e.\ on $\CC\setminus A$, then
\begin{equation*}
|F(A)|\le K|A|.
\end{equation*}
\item Hence, in general,
\begin{equation*}
|F(A)|\le K\pi^{1-1/K}|A|^{1/K}.
\end{equation*}
\end{enumerate}
Here again $K=(1+k)/(1-k)$.
\end{theorem}

Theorem~\ref{T:areadistortion} is a sharpened form of a result of Astala \cite[Theorem~1]{As94} due to Eremenko and Hamilton \cite[Theorem~1]{EH95}.

We also show how the proof of Theorem \ref{T:Hausdorff} can be adapted to obtain the following upper bound for the Hausdorff dimension of quasicircles due to Smirnov \cite{Sm10}.

\begin{theorem}\label{T:Smirnov}
If $\Gamma$ is a  $k$-quasicircle, then $\dim_H(\Gamma)\le1+k^2$.
\end{theorem}

Finally, we  obtain a result on the distortion of dimension under quasi\-symmetric maps. For the Hausdorff dimension $\dim_H$, it was proved by Prause and Smirnov, see the main result of \cite{PS11} and also \cite[Theorem 3.1]{Pr07}.
In the theorem below, $\dim$ denotes one of $\overline{\dim}_M$ or
$\dim_P$. In the case of $\overline{\dim}_M$, we also assume that $A$ is bounded.

\begin{theorem}\label{T:PrauseSmirnov}
Let $g:\RR \to \RR$ be a $k$-quasisymmetric map, where $k \in [0,1)$.
Then, given a  set $A \subset \RR$ with $\dim(A)=\delta$, $0<\delta \le 1$, we have
\[
\Delta(\delta,k) \le \dim(g(A)) \le \Delta^*(\delta,k).
\]
Here
\[
\Delta(\delta,k):= 1 - \left( \frac{k+l}{1+kl} \right)^2
\]
where $l:=\sqrt{1-\delta}$, and $\Delta^*(\delta,k)$ is the inverse \[
\Delta^*(\delta,k):= \Delta(\delta, - \min(k,\sqrt{1-\delta})).
\]
In particular, if $\dim A=\delta=1$, then $l=0$ and $\Delta(\delta,k)=1-k^2$, whence
\[
\dim(g(A)) \ge 1-k^2.
\]
\end{theorem}

The remainder of the paper is organized as follows. We review the notions of Hausdorff, packing and Minkowski dimensions in \S\ref{S:dimension}.
In \S\ref{S:holomotions} we discuss  holomorphic motions in more detail,
in particular their relation to quasiconformal maps.
The basic properties of inf-harmonic functions that we need are developed
in \S\ref{S:infharmonic}.
Our main results, Theorems~\ref{T:Minkowski},
\ref{T:packing}, \ref{T:Hausdorff}, \ref{T:converse} and \ref{T:Area},
are proved in \S\S\ref{S:Minkowski}--\ref{S:Area}.
The applications to quasiconformal mappings,
namely Theorems \ref{T:qcdistortion}, \ref{T:areadistortion}, \ref{T:Smirnov} and
\ref{T:PrauseSmirnov}, are treated
in \S\ref{S:quasiconformal}. We conclude in \S\ref{S:conclusion}
with an open problem.

%%%%%%%%%%%%%%%%%%%%%%%%%%%%%%%%%%%%%%%%%%%%%%%%%%

\section{Notions of dimension}\label{S:dimension}

In this section we present a very brief review of some basic  notions
of  dimension,
introducing the  notation, and
concentrating on the aspects that will be useful to us later.
Our account is based on the books of Bishop--Peres \cite{BP17} and  Falconer \cite{Fa14}.

\subsection{Hausdorff dimension}

We begin with the definition. Let $A\subset\CC$.
For $s\ge0$ and $\delta>0$, define
\[
\cH_\delta^s(A):=\inf\Bigl\{\sum_{j=1}^\infty \diam(A_j)^s\Bigr\},
\]
where the infimum is taken over all countable covers $\{A_j\}$ of $A$
by sets of diameter at most $\delta$. Since $\cH^s_\delta(A)$
increases as $\delta$ decreases, the limit
\[
\cH^s(A):=\lim_{\delta\to0} \cH^s_\delta(A)
\]
exists, possibly $0$ or $\infty$.
The set function $\cH^s(\cdot)$ is an outer measure on~$\CC$,
called the \emph{$s$-dimensional Hausdorff measure}.
The \emph{Hausdorff dimension} of $A$ is  defined as the unique real number $\dim_H(A)\in[0,2]$ such that
\[
\cH^s(A)=
\begin{cases}
\infty, &s<\dim_H(A),\\
0, &s>\dim_H(A).
\end{cases}
\]

We shall need a slight variant of this construction.
A \emph{dyadic square} is a subset  of $\CC$ of the form
$Q=[m2^{-k}, (m+1)2^{-k})\times[n2^{-k},(n+1)2^{-k})$, where $k,m,n$
are integers (possibly negative).
Define
\[
\tilde{\cH}_\delta^s(A):=\inf\Bigl\{\sum_{j=1}^\infty \diam(Q_j)^s\Bigr\},
\]
where now the infimum is taken merely over countable covers
$\{Q_j\}$ of $A$ by dyadic squares of diameter at most $\delta$.
As before, we also set
\[
\tilde{\cH}^s(A):=\lim_{\delta\to0} \tilde{\cH}^s_\delta(A).
\]
Clearly we have $\tilde{\cH}^s_\delta(A)\ge \cH^s_\delta(A)$ for all
$\delta$,
and hence $\tilde{\cH}^s(A)\ge \cH^s(A)$. Also,
it is not hard to see that any bounded subset of $\CC$
can be covered by $9$  dyadic squares of smaller diameter,
from which it follows  that $\tilde{\cH}^s_\delta(A)\le 9\cH^s_\delta(A)$ for all $\delta$, and hence
$\tilde{\cH}^s(A)\le 9\cH^s(A)$.
In particular, we deduce the following result.

\begin{proposition}\label{P:dyadic}
With the above notation, we have
\[
\tilde{\cH}^s(A)=
\begin{cases}
\infty, &s<\dim_H(A),\\
0, &s>\dim_H(A).
\end{cases}
\]
\end{proposition}

Dyadic squares have the property that any two of them are either nested or disjoint. Thus the sets $Q_j$ in the definition of
$\tilde{\cH}^s_\delta(A)$ may be taken to be disjoint. This will
be useful for us later. For more on this, see \cite[\S1.3, p.11]{BP17}.

We conclude by noting that Hausdorff dimension is
\emph{countably stable}, i.e., for any sequence of sets $(A_j)$
we have $\dim_H(\cup_{j\ge1}A_j)=\sup_{j\ge1}\dim_H(A_j)$
(see e.g.\ \cite[p.49]{Fa14}).

\subsection{Packing dimension}

The notion of packing dimension
is in some sense dual to that of Hausdorff dimension.
It was introduced by Tricot in \cite{Tr82}.

Once again, we begin with the definition.
Let $A\subset\CC$. For $s\ge0$ and $\delta>0$, define
\[
\cP^s_\delta(A):=\sup\Bigl\{\sum_{j=1}^n \diam(D_j)^s\Bigr\},
\]
where the supremum is taken over all finite sets of disjoint disks $\{D_j\}$
with centres in $A$ and of diameters at most $\delta$.
Since $P_\delta^s(A)$ decreases as $\delta$ decreases, the limit
\[
\cP_0^s(A):=\lim_{\delta\to0}\cP_\delta^s(A)
\]
exists, possibly $0$ or $\infty$.
This is not yet an outer measure, because it is not countably subadditive.
It is sometimes called
the \emph{$s$-dimensional pre-packing measure} of $A$. We
modify it to make it an outer measure, defining the
\emph{$s$-dimensional packing measure} of $A$ by
\[
\cP^s(A):=\inf\Bigl\{\sum_{j\ge1}\cP_0^s(A_j):A=\cup_{j\ge1}A_j\Bigr\},
\]
where the infimum is taken over all countable covers of $A$ by
subsets $(A_j)_{j\ge1}$.
The \emph{packing dimension} of $A$
is then
defined as the unique real number $\dim_P(A)\in[0,2]$ such that
\[
\cP^s(A)=
\begin{cases}
\infty, &s<\dim_P(A),\\
0, &s>\dim_P(A).
\end{cases}
\]

As in the case of Hausdorff dimension, the packing dimension is countably stable:
$\dim_P(\cup_{j\ge1}A_j)=\sup_{j\ge1}\dim_P(A_j)$. Also, we always have
\[
\dim_H(A)\le\dim_P(A),
\]
and the inequality may be strict.

\subsection{Minkowski dimension}

Let $A$ be a bounded subset of $\CC$.
Given $\delta>0$, we denote by $N_\delta(A)$  the
smallest number of sets of diameter at most~$\delta$ needed to cover $A$.
The \emph{upper} and \emph{lower Minkowski dimensions}
of $A$ are respectively defined by
\[
\overline{\dim}_M(A):=\limsup_{\delta\to0}\frac{\log N_\delta(A)}{\log(1/\delta)}
\quad\text{and}\quad
\underline{\dim}_M(A):=\liminf_{\delta\to0}\frac{\log N_\delta(A)}{\log(1/\delta)}.
\]
Of course we always have $\underline{\dim}_M(A)\le \overline{\dim}_M(A)$. The inequality may be strict.
If equality holds, then we speak simply of the Minkowski dimension
of $A$, denoted $\dim_M(A)$. It is also  called the
\emph{box-counting dimension} of $A$.

The Minkowski dimension has the virtue of simplicity,
but it also suffers from the drawback that,
unlike the Hausdorff and packing dimensions,
it is not countably stable, i.e., it can happen that
$\dim_M(\cup_jA_j)>\sup_j\dim_M(A_j)$.

There is a useful relationship between upper Minkowski
dimension and the pre-packing measure $\cP_0^s$
introduced in the previous subsection.
The following result is due to Tricot \cite[Corollary~2]{Tr82}.

\begin{proposition}\label{P:Tricot}
If $A$ is a bounded subset of $\CC$, then
\[
\cP_0^s(A)=
\begin{cases}
\infty, &s<\overline{\dim}_M(A),\\
0, &s>\overline{\dim}_M(A).
\end{cases}
\]
\end{proposition}

Using this result, we can express the packing dimension
in terms of the upper Minkowski dimension.
The following theorem is again due to Tricot \cite[Proposition~2]{Tr82}, see also \cite[Theorem~2.7.1]{BP17}.

\begin{proposition}\label{P:Minkowskipacking}
If $A$ is a subset of $\CC$, then
\[
\dim_P(A)=\inf\Bigl\{\sup_{j\ge1}\overline{\dim}_M(A_j):A=\cup_{j\ge1}A_j\Bigr\},
\]
where the infimum is taken over all countable covers
of $A$ by bounded subsets $(A_j)$.
\end{proposition}

From this result, it is obvious that, for every bounded set $A$,
we have
\[
\dim_P(A)\le \overline{\dim}_M(A).
\]
In general the inequality can be strict.
The books \cite{BP17} and \cite{Fa14} both contain a discussion of
conditions under which equality holds.

\subsection{Similarity dimension}
There is one further notion of dimension that will prove useful
in what follows. It applies to a specific example.

Consider a finite system of contractive similarities
\[
\gamma_{j}(z)=a_jz+b_j
\quad(j=1,\dots,n),
\]
where $a_1,\dots,a_n,b_1,\dots,b_n\in\CC$  and $|a_j|<1$ for all $j$.
In this situation, there is a unique compact subset $L$ of $\CC$
such that $L=\cup_{j=1}^n \gamma_j(L)$,  called the
\emph{limit set} of the iterated function system $\{\gamma_1,\dots,\gamma_n\}$.

The system $\{\gamma_1,\dots,\gamma_n\}$
is said to satisfy the \emph{open set condition}
if there exists a non-empty open subset $U$ of $\CC$ such that
$\gamma_j(U)\subset U$ for all $j$ and
$\gamma_i(U)\cap\gamma_j(U)=\emptyset$ whenever $i\ne j$.
The following result is due to Hutchinson~\cite{Hu81},
generalizing an earlier result of Moran \cite{Mo46},
see also \cite[Theorem~9.3]{Fa14} or \cite[Theorem~2.2.2]{BP17}.

\begin{theorem}\label{T:simdim}
If the system $\{\gamma_1,\dots,\gamma_n\}$
satisfies the open set condition,
then the Hausdorff and packing dimensions of its limit set $L$
are given by $\dim_H(L)=\dim_P(L)=s$,
where $s$ is the unique solution of the equation
\[
\sum_{j=1}^n |a_j|^s=1.
\]
\end{theorem}

The number $s$ (with or without the open set condition) is  called the
\emph{similarity dimension} of the system.

%%%%%%%%%%%%%%%%%%%%%%%%%%%%%%%

\section{Holomorphic motions and quasiconformal maps}\label{S:holomotions}
Holomorphic motions were defined in Definition~\ref{D:holomotion}.
As was mentioned in the introduction, they were introduced in \cite{MSS83}
by Ma\~{n}\'{e}, Sad and Sullivan, who also established the $\lambda$-lemma.
Their result was later improved by Slodkowski in \cite{Sl91},
confirming a conjecture of
Sullivan and Thurston  \cite{ST86}.
Slodkowski's result is often called the extended $\lambda$-lemma.
There are now several proofs; another one can be found in
\cite[\S12]{AIM09}.

\begin{theorem}[Extended $\lambda$-lemma]\label{T:extended}
A holomorphic motion $f:\DD\times A\to\CC$
has an extension to a holomorphic motion
$F:\DD\times\CC\to\CC$. The function $F$ is jointly continuous
on $\DD\times\CC$.
\end{theorem}

As was already remarked in \cite{MSS83},
holomorphic motions are closely related to
quasiconformal maps.
We now define this term and state some results that will be needed
in the sequel.
Our treatment follows that in \cite{AIM09}.

\begin{definition}\label{D:quasiconformal}
Let $\Omega,\Omega'$ be plane domains. A homeomorphism
$f:\Omega\to \Omega'$ is called \emph{quasiconformal} if:
\begin{enumerate}[\normalfont(i)]
\item $f$ is orientation-preserving;
\item its distributional Wirtinger derivatives $\partial f/\partial z$ and $\partial f/\partial\overline{z}$ both belong to $L^2_{\loc}(\Omega)$, and
\item $f$ satisfies the \emph{Beltrami equation}:
\[
\frac{\partial f}{\partial \overline{z}}=\mu_f \, \frac{\partial f}{\partial z}\quad\text{a.e. on $\Omega$},
\]
where $\mu_f$ is a measurable function on $\Omega$ such that $\|\mu_f\|_{\infty} <1$.
\end{enumerate}
The function $\mu_f$ is called the \emph{Beltrami coefficient} or
\emph{complex dilatation} of~$f$.
We shall say that the mapping  $f$ is  \emph{$k$-quasiconformal} if
$\|\mu_f\|_\infty\le k$.
\end{definition}

\begin{remark}
Many authors (including those of \cite{AIM09}) use the term $K$-quasi\-con\-formal to mean $k$-quasi\-conformal in our sense with $K=(1+k)/(1-k)$.
\end{remark}

We shall  need the following fundamental result on the existence and uniqueness of solutions to the Beltrami equation
\cite[Theorem~5.3.4]{AIM09}.

\begin{theorem}[Measurable Riemann mapping theorem]\label{T:MRMT}
Let $\mu$ be a measurable function on $\CC$ with $\|\mu\|_\infty <1$. Then there exists a unique quasi\-conformal mapping $f:\CC \to \CC$ fixing $0$ and $1$ with
$\mu_f=\mu$ a.e.\ on $\CC$.
\end{theorem}

It is well known that solutions of the Beltrami equation depend holo\-morphically on the parameter $\mu$ \cite[Corollary~5.7.5]{AIM09}. Combined with \cite[Theorem~12.3.2]{AIM09}, this is the key to the following characterization of holomorphic motions.

\begin{theorem}\label{T:motions}
Let $f:\mathbb{D} \times \CC \to \CC$ be a function. The following statements are equivalent:
\begin{enumerate}[\normalfont(i)]
\item The map $f$ is a holomorphic motion.
\item For each $\lambda \in \mathbb{D}$, the map $f_\lambda : \CC \to \CC$ is quasiconformal with Beltrami coefficient $\mu_\lambda$ satisfying $\|\mu_\lambda\|_\infty \le |\lambda|$. Moreover, the map $f_0$ is the identity, and the $L^\infty(\CC)$-valued map $\lambda \mapsto \mu_\lambda$ is holomorphic on $\mathbb{D}$.
\end{enumerate}
\end{theorem}

These results can be used to show that every quasiconformal homeomorphism of $\CC$ can be embedded as part of a holomorphic motion \cite[Theorem~12.5.3]{AIM09}.

\begin{theorem}\label{T:embedding}
If $F:\CC\to\CC$ is a $k$-quasiconformal homeomorphism,
then there exists a holomorphic motion $f:\DD\times\CC\to\CC$ such that $f_k=F$.
\end{theorem}

Quasiconformal maps exhibit numerous interesting properties.
An important one for us is  the fact that quasiconformal homeomorphisms of $\CC$ are quasisymmetric in the sense
described in the next theorem
\cite[Theorem~3.5.3]{AIM09}.

\begin{theorem}\label{T:quasisym}
Given $k\in[0,1)$, there exists an increasing homeomorphism
$\eta:[0,\infty)\to[0,\infty)$ such that every $k$-quasiconformal
map $f:\CC\to\CC$ satisfies
\begin{equation}\label{E:quasisym}
\frac{|f(z_0)-f(z_1)|}{|f(z_0)-f(z_2)|}\le
\eta\Bigl(\frac{|z_0-z_1|}{|z_0-z_2|}\Bigr)
\quad(z_0,z_1,z_2\in\CC).
\end{equation}
\end{theorem}

We shall exploit this result via the following
simple corollary.

\begin{corollary}\label{C:quasisym}
Given $k\in[0,1)$,
there exist constants $\delta,\delta'>0$
such that every $k$-quasiconformal homeomorphism
$f:\CC\to\CC$ has the following properties:
\begin{enumerate}[\normalfont(i)]
\item If $z_0\in\CC$ and $D$ is an open disk with centre $z_0$,
then $f(D)$ contains the open disk with centre $f(z_0)$
and radius $\delta \diam f(D)$.
\item If $z_0\in\CC$ and $Q$ is an open square with centre $z_0$,
then $f(Q)$ contains the open disk with centre $f(z_0)$
and radius $\delta' \diam f(Q)$.
\end{enumerate}
\end{corollary}

\begin{proof}
Let $\eta$ be the function associated to $k$
by Theorem~\ref{T:quasisym}.
In the case of the square, if $z_1,z_2\in\partial Q$,
then  $|z_0-z_1|/|z_0-z_2|\le\sqrt{2}$, so by \eqref{E:quasisym}
we have $|f(z_0)-f(z_1)|/|f(z_0)-f(z_2)|\le \eta(\sqrt{2})$.
It follows that (ii) holds with $\delta'=1/(2\eta(\sqrt{2}))$.
The proof of (i) is similar, now with $\delta=1/(2\eta(1))$.
\end{proof}

%%%%%%%%%%%%%%%%%%%%%%%%%%%%%%%

\section{Inf-harmonic functions}\label{S:infharmonic}

Recall from Definition~\ref{D:infharmonic}
that a function $u:D\to[0,\infty)$ defined on a plane domain $D$
is inf-harmonic if it is the lower envelope of a family of harmonic
functions. It is inherent in the definition that $u$ is positive, so the
harmonic functions are positive too.
This has the consequence that inf-harmonic functions inherit
several of the good properties of positive harmonic functions.

We begin by showing that inf-harmonic functions satisfy Harnack's inequality.
Recall that, given $\lambda_1,\lambda_2\in D$, there exists $\tau>0$ such that,
for all positive harmonic functions $h$ on $D$
\[
\frac{1}{\tau}
\le \frac{h(\lambda_1)}{h(\lambda_2)}
\le\tau.
\]
The smallest such $\tau$ is called the \emph{Harnack distance} between $\lambda_1,\lambda_2$,
denoted $\tau_D(\lambda_1,\lambda_2)$.
For example,
$\tau_\DD(0,\lambda)=(1+|\lambda|)/(1-|\lambda|)$ for  $\lambda\in \DD$.

\begin{proposition}\label{P:Harnackineq}
Let $u$ be an inf-harmonic function on a domain $D$, and suppose that $u\not\equiv0$.
Then $u(\lambda)>0$ for all $\lambda\in D$ and
\begin{equation}\label{E:Harnackineq}
\frac{1}{ \tau_D(\lambda_1,\lambda_2)}\le \frac{u(\lambda_1)}{u(\lambda_2)}\le \tau_D(\lambda_1,\lambda_2)
\quad(\lambda_1,\lambda_2\in D).
\end{equation}
\end{proposition}

\begin{proof}
For each positive harmonic function $h$ on $D$,
we have
\[
\frac{1}{\tau_D(\lambda_1,\lambda_2)}h(\lambda_2)\le h(\lambda_1)\le \tau_D(\lambda_1,\lambda_2) h(\lambda_2)
\quad(\lambda_1,\lambda_2\in D).
\]
Taking the infimum over all $h$ such that $h\ge u$, we obtain
\[
\frac{1}{\tau_D(\lambda_1,\lambda_2)}u(\lambda_2)\le u(\lambda_1)\le \tau_D(\lambda_1,\lambda_2) u(\lambda_2)
\quad(\lambda_1,\lambda_2\in D).
\]
Since $u\not\equiv0$, this shows that $u(\lambda)>0$
for all $\lambda\in D$,
and \eqref{E:Harnackineq} now follows immediately.
\end{proof}

\begin{corollary}\label{C:Harnackineq}
If $u$ is an inf-harmonic function on $D$,
then it is a continuous superharmonic function on $D$.
\end{corollary}

\begin{proof}
The continuity of $u$ follows from Proposition~\ref{P:Harnackineq},
since $\tau_D$ is continuous on $D\times D$.
As $u$ is the infimum of harmonic functions,
it clearly satisfies the super-mean value property, so it is superharmonic on $D$.
\end{proof}

The next result is a normal-family property.

\begin{proposition}\label{P:normalfamily}
Let $(D_n)_{n\ge1}$ be an increasing sequence of domains,
and let $D:=\cup_{n\ge1}D_n$.
For each $n$, let $u_n$ be an inf-harmonic function on  $D_n$.
Then either $u_n\to\infty$ locally uniformly on $D$,
or else some subsequence $u_{n_j}\to u$ locally uniformly on $D$,
where $u$ is inf-harmonic on $D$.
\end{proposition}

\begin{proof}
If there exists a point $\lambda_0\in D$ such that $u_n(\lambda_0)\to\infty$,
then by Proposition~\ref{P:Harnackineq}
the sequence $u_n\to\infty$ locally uniformly in each $D_m$
and hence also on $D$.
Likewise if $u_n(\lambda_0)\to0$,
then $u_n\to0$ locally uniformly in $D$.
Thus, replacing $(u_n)$ by a subsequence if necessary,
we may assume that
there exists $\lambda_0\in D_1$ and $M>1$ such that
$1/M\le u_n(\lambda_0)\le M$ for all $n$.
In this case, by Proposition~\ref{P:Harnackineq} once more,
the sequence $(u_n)$ is equicontinuous on each $D_m$,
and by the Arzel\`a--Ascoli theorem,
a subsequence $(u_{n_j})$ converges locally uniformly on $D$
to a finite-valued function $u$.

It remains to show that $u$ is itself inf-harmonic on $D$.
Relabelling, if necessary,
we can suppose that the whole sequence $u_n$ converges to $u$
locally uniformly on $D$. Let $\lambda_0\in D$.
Choose $n_0$ so that $\lambda_0\in D_{n_0}$.
For each $n\ge n_0$, the function $u_n$ is inf-harmonic on $D_n$,
so there exists a (positive) harmonic function $h_n$ on $D_n$
such that $h_n\ge u_n$ on $D_n$
and $h_n(\lambda_0)\le u(\lambda_0)+1/n$.
By a standard normal-family argument, a subsequence
$(h_{n_j})$ converges locally uniformly on $D$ to a function
$h$ that is harmonic on $D$.
Clearly $h\ge u$ on $D$ and $h(\lambda_0)=u(\lambda_0)$.
Such an $h$ exists for each choice
of $\lambda_0\in D$, so
we conclude that $u$ is indeed inf-harmonic on $D$.
\end{proof}

The following result lists some closure properties
of the family of inf-harmonic functions.

\begin{proposition}\label{P:closure}
\begin{enumerate}[\normalfont\rm(i)]
\item\label{I:linear} If $u$ and $v$ are inf-harmonic on $D$ and if $\alpha,\beta\ge0$,
then $\alpha u+\beta v$ is inf-harmonic on $D$.
\item\label{I:diff} If $u$ is inf-harmonic on $D$ and $h$ is harmonic on $D$,
and if $u\ge h$, then $u-h$ is inf-harmonic on $D$.
\item\label{I:limit} If $(u_n)_{n\ge1}$ are inf-harmonic functions on $D$,
and if $u_n\to u$ pointwise on $D$,
then either $u$ is inf-harmonic on $D$ or $u\equiv\infty$.
\item\label{I:semilocal} If $(D_n)$ is an increasing sequence of domains
with $\cup_{n\ge1}D_n=D$,
and if $u$ is a function on $D$ such that
$u|_{D_n}$ is inf-harmonic on $D_n$ for each $n$,
then $u$ is inf-harmonic on $D$.
\item\label{I:inf} If $\cV$ is a family of inf-harmonic functions on $D$
and $u:=\inf_{v\in\cV}v$,
then $u$ is inf-harmonic on $D$.
\item\label{I:sup} If $\cV$ is an upward-directed family of inf-harmonic functions on $D$
(i.e., given $v_1,v_2\in\cV$, there exists $v_3\in\cV$
with $v_3\ge \max\{v_1,v_2\}$), and if $u:=\sup_{v\in\cV}v$, then
either $u$ is inf-harmonic on $D$ or $u\equiv\infty$.
\end{enumerate}
\end{proposition}

\begin{proof}
\ref{I:linear},\ref{I:diff} These are both obvious.

\ref{I:limit} Assume that $u\not\equiv\infty$.
Then, by Proposition~\ref{P:normalfamily},
a subsequence of the $(u_n)$ converges locally uniformly on $D$
to an inf-harmonic function $v$.
Since the same subsequence converges pointwise to $u$,
we must have $v=u$. Hence $u$ is inf-harmonic.

\ref{I:semilocal} This follows by applying Proposition~\ref{P:normalfamily}
with $u_n:=u|_{D_n}$.

\ref{I:inf} Again, this is obvious.

\ref{I:sup} Assume that $u\not\equiv\infty$.
Then, by Proposition~\ref{P:Harnackineq},
$u$ is finite-valued and continuous on $D$.
Let $\Lambda=(\lambda_j)$ be a sequence that is dense in $D$.
Using the fact that $\cV$ is upward-directed,
we may construct an increasing sequence of functions $v_n\in\cV$ such that
$v_n(\lambda_j)\ge u(\lambda_j)-1/n$ for all $j\in\{1,2,\dots,n\}$ and all $n\ge1$.
Then $v_n$ converges pointwise to a function $v$ such that
$v\le u$ and $v=u$ on $\Lambda$.
By part~\ref{I:limit} above, $v$ is inf-harmonic on $D$.
As $v=u$ on the dense subset $\Lambda$ and both
$u,v$ are continuous, we have $v=u$ on $D$.
Thus $u$ is inf-harmonic on $D$, as asserted.
\end{proof}

We conclude this section with an implicit function theorem
for inf-harm\-onic functions.

\begin{theorem}\label{T:implicit}
Let $D$ be a plane domain,
and let $a_j:D\to(0,1)$ be
a finite or infinite sequence of functions such that
$\log (1/a_j)$ is inf-harmonic on $D$ for each $j$.
Let $c>0$, and define $s:D\to[0,\infty]$ by
\[
s(\lambda):=\inf\Bigl\{ \alpha> 0 :\sum_j a_j(\lambda)^\alpha\le c\Bigr\}
\quad(\lambda\in D),
\]
where we interpret $\inf\emptyset=\infty$.
Then either $s\equiv0$ or $1/s$ is an inf-harmonic function on $D$.
\end{theorem}

It is perhaps worth emphasizing the case where there are
only finitely many functions $a_j$. It then becomes a result
closely linked to the notion of similarity dimension defined
in \S\ref{S:dimension}.
It generalizes  a result of Baribeau and Roy \cite[Theorem 1]{BR06}.

\begin{corollary}\label{C:implicit}
Let $a_1,\dots,a_n:D\to(0,1)$ be functions such that
$\log (1/a_j)$ is inf-harmonic on $D$ for each $j$.
Let $c\in(0,n)$ and, for each $\lambda\in D$, let $s(\lambda)$
be the unique solution of the equation
\[
\sum_{j=1}^na_j(\lambda)^{s(\lambda)}=c.
\]
Then  $1/s$ is an inf-harmonic function on $D$.
\end{corollary}

We shall deduce Theorem~\ref{T:implicit}
from a more general abstract result.
To formulate this result, it is convenient to introduce
some terminology.

Let $X$ be a  set and let $\cU$ be a  family
of functions $u:X\to[0,\infty)$.
We call $\cU$ an \emph{inf-cone} on $X$
if it satisfies the following closure properties:
\begin{itemize}
\item if $u,v\in\cU$ and $\alpha,\beta\ge0$, then $\alpha u+\beta v\in\cU$;
\item if $\emptyset\ne\cV\subset\cU$ and $u:=\inf_{v\in\cV}v$, then $u\in\cU$.
\end{itemize}
By Proposition~\ref{P:closure} parts~\ref{I:linear} and~\ref{I:inf},
the set of inf-harmonic functions on a domain $D$ is an inf-cone on $D$.

The following result may be viewed as an abstract implicit function theorem for inf-cones.

\begin{lemma}\label{L:implicit}
Let $\cU$ be an inf-cone on $X$, let $(u_j)_{j\ge1}$ be a sequence in $\cU$,
and for each $j$ let $\phi_j:[0,\infty)\to[0,\infty)$ be a continuous, decreasing, convex function.
Define $v:X\to[0,\infty]$ by
\[
v(x):=\sup\Bigl\{t>0: \sum_{j\ge1}\phi_j(u_j(x)/t)\le1\Bigr\}
\quad(x\in X),
\]
where we interpret $\sup\emptyset=0$.
Then $v\in\cU$ or $v\equiv\infty$.
\end{lemma}

\begin{proof}
For each $j$, let $\cL_j$ be the family of functions of the form $L(y):=b_L-a_Ly$,
such that $a_L\ge0,b_L\in\RR$ and $L\le\phi_j$.
As $\phi_j$ is a continuous decreasing convex function, we have $\phi_j=\sup_{L\in\cL_j}L$.
Consequently, if $x\in X$ and $t>0$, then
\begin{align*}
&\sum_{j\ge1}\phi_j(u_j(x)/t)\le1\\
&\iff \sum_{j=1}^n\phi_j(u_j(x)/t)\le 1\quad(n\ge1)\\
&\iff \sum_{j=1}^n (b_{L_j}-a_{L_j}u_j(x)/t)\le 1 \quad(n\ge1,  ~L_1\in\cL_1,~\dots,~L_n\in\cL_n)\\
&\iff t\Bigl(\sum_{j=1}^n b_{L_j}-1\Bigr)\le \sum_{j=1}^n a_{L_j}u_j(x)\quad(n\ge1,  ~L_1\in\cL_1,~\dots,~L_n\in\cL_n).
\end{align*}
There are now two possibilities. If $\sum_{j=1}^n b_{L_j}\le1$ for all  $n$
and all choices of $(L_1,\dots,L_n)\in\cL_1\times\dots\times\cL_n$, then
the above conditions are satisfied for all $t>0$ and all $x\in X$. In this case $v\equiv\infty$.
In the other case, we have
\[
v(x)=\inf\Biggl\{\frac{\sum_{j=1}^n a_{L_j}u_j(x)}{\sum_{j=1}^n b_{L_j}-1}\Biggr\}
\quad(x\in X),
\]
where the infimum is taken over all $n\ge1$ and all $(L_1,\dots,L_n)\in\cL_1\times\cdots\times\cL_n$
such that  $\sum_{j=1}^n b_{L_j}>1$.
Hence $v\in\cU$ in this case.
\end{proof}

\begin{proof}[Proof of Theorem~\ref{T:implicit}]
This result  follows from Lemma~\ref{L:implicit}
upon taking $\cU$ to be the set of inf-harmonic functions on $D$,
and $\psi_j(y):=(1/c)\exp(-y)$ for each $j$.
\end{proof}

%%%%%%%%%%%%%%%%%%%%%%%%%%%%%%%%
%%%%%%%%%%%%%%%%%%%%%%%%%%%%%%%%

\section{Proof of Theorem~\ref{T:Minkowski}}\label{S:Minkowski}

We have $A_\lambda=f_\lambda(A)=f(\lambda,A)$,
where $f:\DD\times A\to\CC$ is a holomorphic motion.
By Theorem~\ref{T:extended},
we may extend $f$ to a holomorphic motion $f:\DD\times\CC\to\CC$.
We shall assume that $f$ has been so extended.
Since $A$ is bounded and $f$ is continuous,
it follows that $A_\lambda$ is bounded for all $\lambda\in\DD$.

The following lemma establishes the link with inf-harmonic functions.
We recall that $D(a,r)$ denotes the open disk with centre $a$ and radius $r$, and that $\diam(S)$ denotes the euclidean diameter of $S$.

\begin{lemma}\label{L:diameter}
Let $f:\DD\times\CC\to\CC$ be a holomorphic motion.
Let $B$ be a bounded subset of $\CC$ and let $\rho\in(0,1)$.
Then  $M:=\diam f({D}(0,\rho)\times B)<\infty$.
If $S$ is a subset of $B$,
then the map $\lambda\mapsto\log(M/\diam f_\lambda(S))$ is an
inf-harmonic function on $D(0,\rho)$. Consequently, we have
\begin{equation}\label{E:diameter}
\frac{\rho-|\lambda|}{\rho+|\lambda|}
\le\frac{\log(M/\diam f_\lambda(S))}{\log(M/\diam S)}\le
\frac{\rho+|\lambda|}{\rho-|\lambda|}
\qquad(\lambda\in D(0,\rho)).
\end{equation}
\end{lemma}

\begin{proof}
As $f:\DD\times\CC\to\CC$ is a continuous map
and $\overline{D}(0,\rho)\times \overline{B}$ is a
compact subset of $\DD\times\CC$, it follows that
$f(\overline{D}(0,\rho)\times \overline{B})$ is a compact subset
of $\CC$. In particular it has finite diameter, so $M<\infty$.

Given $S\subset B$,  we have
\[
\log\Bigl(\frac{M}{\diam f_\lambda(S)}\Bigr)=
\inf\Bigl\{\log\Bigl(\frac{M}{|f_\lambda(z)-f_\lambda(w)|}\Bigr):z,w\in S,\, z\ne w\Bigr\}.
\]
Fo each pair $z,w\in S$ with $z\ne w$, the function
$\lambda\mapsto\log(M/|f_\lambda(z)-f_\lambda(w)|)$
is positive and harmonic on $D(0,\rho)$.
Therefore $\lambda\mapsto\log(M/\diam f_\lambda(S))$
is inf-harmonic on $D(0,\rho)$.

Finally, the inequality \eqref{E:diameter} is a direct
consequence of Harnack's inequality for inf-harmonic functions,
Proposition~\ref{P:Harnackineq}.
\end{proof}

The next lemma contains the heart of the proof of
Theorem~\ref{T:Minkowski}.
We recall that the upper Minkowski dimension
$\overline{\dim}_M$
can be characterized using Proposition~\ref{P:Tricot}.

\begin{lemma}\label{L:heartM}
If $\overline{\dim}_M(A)>0$, then there exists an inf-harmonic function
$u$ on~$\DD$ such that
\[
u(0)=1/\overline{\dim}_M(A)
\qquad\text{and}\qquad
u(\lambda)\ge1/\overline{\dim}_M(A_\lambda)\quad(\lambda\in\DD).
\]
\end{lemma}

\begin{proof}
Let $\rho\in(0,1)$. We shall carry out the proof on the disk
$D(0,\rho)$, and then let $\rho\to1$ at the very end.

Let $(d_n)$ be a sequence  such that $0<d_n<\overline{\dim}_M(A)$
and $d_n\to\overline{\dim}_M(A)$.
By Proposition~\ref{P:Tricot}, for each $n$ there exists
a finite set $\cD_n$ of disjoint disks with centres in $A$
such that, as $n\to\infty$,
\begin{equation}\label{E:limits}
\max_{D\in\cD_n}\diam(D)\to0
\quad\text{and}\quad
\sum_{D\in\cD_n}\diam(D)^{d_n}\to\infty.
\end{equation}

Let $B$ be the union of all the disks in $\cup_{n\ge1}\cD_n$.
This is a bounded set, so, by Lemma~\ref{L:diameter},
$M:=\diam f(D(0,\rho)\times B)<\infty$, and
$\lambda\mapsto\log(M/\diam f_\lambda(D))$
is inf-harmonic on $D(0,\rho)$ for each $D\in\cup_n\cD_n$.

For each $\lambda\in D(0,\rho)$, let $s_n(\lambda)$
be the unique solution of the equation
\[
\sum_{D\in\cD_n}(\diam f_\lambda(D)/M)^{s_n(\lambda)}=
\sum_{D\in\cD_n}(\diam D/M)^{d_n}.
\]
Clearly $s_n(0)=d_n$.
Also, by the implicit function theorem,
Corollary~\ref{C:implicit},
the function $1/s_n$ is inf-harmonic
on $D(0,\rho)$.
By Proposition~\ref{P:normalfamily}, a subsequence of $1/s_n$
(which, by relabelling, we may suppose to be the whole sequence)
converges locally uniformly  to an inf-harmonic function $u$
on $D(0,\rho)$. Clearly
we have $u(0)=\lim_n(1/d_n)=1/\overline{\dim}_M(A)$.
We shall show that $u(\lambda)\ge 1/\overline{\dim}_M(A_\lambda)$
for all $\lambda\in D(0,\rho)$.

Fix $\lambda\in D(0,\rho)$,
and let $c\in(0,1/u(\lambda))$.
Then $s_n(\lambda)>c$ for all large enough $n$,
and so, for these $n$, we have
\begin{align*}
\sum_{D\in\cD_n}(\diam f_\lambda(D)/M)^c
&\ge\sum_{D\in\cD_n}(\diam f_\lambda(D)/M)^{s_n(\lambda)}\\
&=\sum_{D\in\cD_n}(\diam D/M)^{d_n},
\end{align*}
whence
\[
\sum_{D\in\cD_n}(\diam f_\lambda(D))^c
\ge M^{c-d_n}\sum_{D\in\cD_n}(\diam D)^{d_n}.
\]
For a given value of $n$, the sets $\{f_\lambda(D):D\in\cD_n\}$
are disjoint, but they are not disks. However, we can circumvent
this difficulty by invoking the theory of quasiconformal mappings.
By Theorem~\ref{T:motions},
the map $f_\lambda$ is a
$\rho$-quasiconformal self-homeomorphism of $\CC$.
Consequently, by Corollary~\ref{C:quasisym}(i),
there exists a $\delta>0$ such that,  for each $w\in\CC$
and each open disk $D$ with centre $w$, the set $f_\lambda(D)$
contains the open disk with centre $f_\lambda(w)$ and
radius $\delta \diam f_\lambda(D)$.
In particular, for each $D\in\cD_n$, the set $f_\lambda(D)$
contains a disk with centre in $f_\lambda(A)$ and diameter
at least $\delta\diam f_\lambda(D)$. Denoting by $\cD_n'$
the set of such disks,
we obtain a finite set of disjoint disks $D'$ with centres in $A_\lambda$
and such that
\[
\sum_{D'\in\cD_n'}(\diam D')^c
\ge \sum_{D\in\cD_n}(\delta\diam f_\lambda(D))^c
\ge \delta^cM^{c-d_n}\sum_{D\in\cD_n}(\diam D)^{d_n}.
\]
From \eqref{E:limits},
we have $\sum_{D\in\cD_n}(\diam D)^{d_n}\to\infty$,
whence it follows that
\begin{equation}\label{E:sumlim}
\sum_{D'\in\cD_n'}(\diam D')^c\to\infty \quad(n\to\infty).
\end{equation}
Also from \eqref{E:limits},
we have $\max_{D\in\cD_n}\diam(D)\to0$,
which, together with the inequality \eqref{E:diameter}, implies that
\begin{equation}\label{E:maxlim}
\max_{D'\in\cD_n'}\diam(D')\to0 \quad(n\to\infty).
\end{equation}
Taken together, the limits \eqref{E:sumlim} and \eqref{E:maxlim} show that
$\overline{\dim}_M(A_\lambda)\ge c$.
As this holds for each $c\in(0,1/u(\lambda))$, we deduce that
$\overline{\dim}_M(A_\lambda)\ge 1/u(\lambda)$,
in other words, that $u(\lambda)\ge 1/\overline{\dim}_M(A_\lambda)$,
as desired.

The proof of the lemma is nearly complete,
save for the fact that $u$ is defined only on $D(0,\rho)$,
not on $\DD$. To fix this, let us choose an increasing sequence
$(\rho_m)$ in $(0,1)$ such that $\rho_m\to1$.
For each $m$, the argument above furnishes an
inf-harmonic function $u_m$ defined on $D(0,\rho_m)$
such that $u_m(0)=1/\overline{\dim}_M(A)$
and $u_m(\lambda)\ge 1/\overline{\dim}_M(A_\lambda)$ for all $\lambda\in D(0,\rho_m)$.
By Proposition~\ref{P:normalfamily}, a subsequence of $(u_m)$
converges locally uniformly to an inf-harmonic function $u$ on $\DD$.
Clearly we have $u(0)=1/\overline{\dim}_M(A)$
and $u(\lambda)\ge 1/\overline{\dim}_M(A_\lambda)$
for all $\lambda\in \DD$.
The proof  is now complete.
\end{proof}

From here, it is a small step to establish  the main result.

\begin{proof}[Proof of Theorem~\ref{T:Minkowski}]
It is enough to show that, for each $\lambda_0\in\DD$
such that $\overline{\dim}_M(A_{\lambda_0})>0$, there exists
an inf-harmonic function $u$ on $\DD$ such that
\begin{equation}\label{E:uineq}
u(\lambda_0)=1/\overline{\dim}_M(A_{\lambda_0})
\qquad\text{and}\qquad
u(\lambda)\ge1/\overline{\dim}_M(A_\lambda)\quad(\lambda\in\DD).
\end{equation}
The special case $\lambda_0=0$ has already been proved in Lemma~\ref{L:heartM}.
The general case can be deduced from this as follows.

Fix a M\"obius automorphism $\phi$ of $\DD$
such that $\phi(0)=\lambda_0$.
Then $\tilde{f}_\lambda:=f_{\phi(\lambda)}\circ f_{\lambda_0}^{-1}$
is a holomorphic motion mapping $\DD\times\CC$ into $\CC$.
Also $\tilde{A}:=A_{\lambda_0}$ is a bounded subset of $\CC$, such that
$\tilde{f}_\lambda(\tilde{A})=A_{\phi(\lambda)}$ for all $\lambda\in\DD$.
Thus, applying Lemma~\ref{L:heartM} with $A,f$ replaced by the
pair $\tilde{A},\tilde{f}$, we deduce that there exists an
inf-harmonic function $v$ on $\DD$ such that
\[
v(0)=1/\overline{\dim}_M(A_{\phi(0)})
\qquad\text{and}\qquad
v(\lambda)\ge 1/\overline{\dim}_M(A_{\phi(\lambda)})\quad(\lambda\in\DD).
\]
Then $u:=v\circ\phi^{-1}$ is an inf-harmonic function on $\DD$
satisfying \eqref{E:uineq}. This completes the proof of the theorem.
\end{proof}

%%%%%%%%%%%%%%%%%%%%%%%%%%%%%%%%%
\section{Proof of Theorem~\ref{T:packing}}\label{S:packing}

As in the previous section, we may suppose that
$A_\lambda=f_\lambda(A)=f(\lambda, A)$,
where $f:\DD\times\CC\to\CC$
is a  holomorphic motion.

We shall deduce Theorem~\ref{T:packing}
from Theorem~\ref{T:Minkowski},
using the characterization of packing dimension
in terms of Minkowski dimension given in
Proposition~\ref{P:Minkowskipacking}.
From that result, we have
\[
\dim_P(A)=\inf\Bigl\{\sup_{j\ge1}\overline{\dim}_M(A_j):A=\cup_{j\ge1}A_j\Bigr\},
\]
where the infimum is taken over all countable covers of
$A$ by bounded subsets~$(A_j)$. Since $f_\lambda$
is a bijection of $A$ onto $A_\lambda$, it follows that,
\[
\dim_P(A_\lambda)=\inf\Bigl\{\sup_{j\ge1}\overline{\dim}_M(f_\lambda(A_j)):A=\cup_{j\ge1}A_j\Bigr\}
\quad(\lambda\in\DD),
\]
and hence
\begin{equation}\label{E:infsup}
\frac{1}{\dim_P(A_\lambda)}
=\sup\Bigl\{\inf_{j\ge1}\frac{1}{\overline{\dim}_M(f_\lambda(A_j))}:A=\cup_{j\ge1}A_j\Bigr\}
\quad(\lambda\in\DD).
\end{equation}

Let $A=\cup_{j\ge1}A_j$ be a countable cover of $A$
by bounded subsets of $A$.
By Theorem~\ref{T:Minkowski},
for each $j$, either $\overline{\dim}_M(f_\lambda(A_j))\equiv0$
or $\lambda\mapsto1/\overline{\dim}_M(f_\lambda(A_j))$
is an inf-harmonic function on $\DD$. It follows that
either $\overline{\dim}_M(f_\lambda(A_j))\equiv0$ for all $j\ge1$
or else $\lambda\mapsto\inf_{j\ge1}1/\overline{\dim}_M(f_\lambda(A_j))$
is an inf-harmonic function on $\DD$.
In the first case, \eqref{E:infsup} implies that
$\dim_P(A_\lambda)\equiv0$. In the second case,
the relation \eqref{E:infsup} expresses $1/\dim_P(A_\lambda)$
as the supremum of a family of inf-harmonic functions.

Ordinarily, the supremum of a family of inf-harmonic functions
is no longer inf-harmonic.
However, this particular family is an upward-directed set,
in the sense of Proposition~\ref{P:closure}\,\ref{I:sup}.
Indeed, given any two countable covers
$A=\cup_iA_i=\cup_jB_j$ of $A$
by bounded sets, there is a third such cover,
namely $A=\cup_{i,j}(A_i\cap B_j)$, with the property that
\[
\sup_{i,j}\overline{\dim}_M(A_i\cap B_j)\le
\min\Bigl\{\sup_i\overline{\dim}_M(A_i),\,\sup_j\overline{\dim}_M(B_j)\Bigr\},
\]
which implies upward-directedness in \eqref{E:infsup}.
By Proposition~\ref{P:closure}\,\ref{I:sup},
it follows that either $\dim_P(A_\lambda)\equiv0$
or  $\lambda\mapsto1/\dim_P(A_\lambda)$
is inf-harmonic on $\DD$.
This completes the proof of Theorem~\ref{T:packing}.\qed

%%%%%%%%%%%%%%%%%%%%%%%%%%%%%%%%

\section{Proof of Theorem~\ref{T:Hausdorff}}\label{S:Hausdorff}

The proof of Theorem~\ref{T:Hausdorff} follows a similar pattern to that of Theorem~\ref{T:Minkowski}, presented in \S\ref{S:Minkowski},
except that, because Hausdorff dimension is defined in terms of coverings rather than packings, some of the inequalities go in the
other direction. Unfortunately, this leads ultimately to a weaker result.

We have $A_\lambda=f_\lambda(A)=f(\lambda,A)$,
where $f:\DD\times A\to\CC$ is a holomorphic motion.
As before, we may extend $f$ to a holomorphic motion
$f:\DD\times\CC\to\CC$, and we shall assume that $f$
has been so extended.

The core of the proof is contained in the following lemma.

\begin{lemma}\label{L:heartH}
\begin{enumerate}[\normalfont(i)]
\item If $\dim_H(A)=0$, then $\dim_H(A_\lambda)=0$ for all $\lambda\in\DD$.
\item If $\dim_H(A)>0$, then there exists an inf-harmonic function
$u$ on~$\DD$ such that
\[
u(0)=1/\dim_H(A)
\quad\text{and}\quad
1/2\le u(\lambda)\le1/\dim_H(A_\lambda)\quad(\lambda\in\DD).
\]
\end{enumerate}
\end{lemma}

\begin{proof}
If $\dim_H(A)=2$, then we may simply take $u\equiv1/2$.
Henceforth, we suppose that $0\le\dim_H(A)<2$.

Let $\rho\in(0,1)$. We shall carry out the proof on the disk
$D(0,\rho)$, and then let $\rho\to1$ at the very end.

Let $(d_n)$ be a sequence  such that $\dim_H(A)<d_n<2$
and $d_n\to\dim_H(A)$.
By Proposition~\ref{P:dyadic}, for each $n$ there exists
a (countable) cover $\cQ_n$ of $A$ by
disjoint dyadic squares
such that, as $n\to\infty$,
\begin{equation}\label{E:limitsQ}
\sup_{Q\in\cQ_n}\diam(Q)\to0
\quad\text{and}\quad
\sum_{Q\in\cQ_n}\diam(Q)^{d_n}\to0.
\end{equation}

We can suppose that all the squares in $\cup_n\cQ_n$ meet $A$.
Thus, if $B$ is the union of all the squares in $\cup_n\cQ_n$,
then $B$ is a bounded set.
By Lemma~\ref{L:diameter},
$M:=\diam f(D(0,\rho)\times B)<\infty$, and
$\lambda\mapsto\log(M/\diam f_\lambda(Q))$
is inf-harmonic on $D(0,\rho)$ for each $Q\in\cup_n\cQ_n$.

Fix a constant $C$, to be chosen later (it will depend only
on  $\rho$), and, for each $\lambda\in D(0,\rho)$,
set
\[
s_n(\lambda):=
\inf\Bigl\{ \alpha> 0 :\sum_{Q\in\cQ_n}
\Bigl(\frac{\diam f_\lambda(Q)}{M}\Bigr)^\alpha\le C\Bigr\}.
\]
By the implicit function theorem,
Theorem~\ref{T:implicit},
either $s_n\equiv0$ or
$1/s_n$ is inf-harmonic
on $D(0,\rho)$.
By Proposition~\ref{P:normalfamily}, a subsequence of $(s_n)$
(which, by relabelling, we may suppose to be the whole sequence)
converges locally uniformly to $s$ on $D(0,\rho)$,
where either   $s\equiv0$ or $1/s$ is inf-harmonic on $D(0,\rho)$.

From \eqref{E:limitsQ} we have
we have $s_n(0)\le d_n$ for all sufficiently large $n$,
so
\begin{equation}\label{E:s(0)}
s(0)=\lim_{n\to\infty}s_n(0)\le \lim_{n\to\infty}d_n=\dim_H(A).
\end{equation}

If $\alpha>s(\lambda)$ for some $\lambda\in D(0,\rho)$,
then $\alpha>s_n(\lambda)$ for all large enough~$n$,
so, for these $n$,
\[
\sum_{Q\in\cQ_n}\diam f_\lambda(Q)^\alpha\le CM^\alpha.
\]
For each $n$, the family $\{f_\lambda(Q):Q\in\cQ_n\}$ is a cover of $A_\lambda$.
Also, from \eqref{E:limitsQ} and \eqref{E:diameter}, we have
$\sup_{Q\in\cQ_n}\diam f_\lambda(Q)\to0$ as $n\to\infty$.
It follows from the definition of Hausdorff dimension
that $\dim_H(A_\lambda)\le\alpha$.
As this holds for each $\alpha>s(\lambda)$, we conclude that
\begin{equation}\label{E:s(lambda)}
s(\lambda)\ge \dim_H(A_\lambda)
\quad(\lambda\in D(0,\rho)).
\end{equation}

Next we show that, if the constant $C$ is chosen sufficiently large,
then we also have $s(\lambda)\le 2$ for all $\lambda\in D(0,\rho)$.
To achieve this, we once again
invoke the theory of quasiconformal mappings.
By Theorem~\ref{T:motions},
the map $f_\lambda$ is a
$\rho$-quasiconformal self-homeomorphism of $\CC$.
Consequently, by Corollary~\ref{C:quasisym}\,(ii),
there exists a constant $\delta'>0$, depending only on $\rho$,
such that,
for each open square $Q$ in $\CC$, the set $f_\lambda(Q)$ contains an open disk of
radius $\delta' \diam f_\lambda(Q)$.
In particular, for each $n$,
the disjoint sets $\{f_\lambda(Q):Q\in\cQ_n\}$ contain
disjoint disks of radii $\delta'\diam f_\lambda(Q)$.
As these disks are all contained within the set $f(D(0,\rho)\times B)$,
which has diameter $M$, consideration of their areas
leads to the inequality
\[
\sum_{Q\in\cQ_n}\pi(\delta'\diam f_\lambda(Q))^2\le \pi M^2,
\]
in other words,
\[
\sum_{Q\in\cQ_n}\Bigl(\frac{\diam(f_\lambda(Q))}{M}\Bigr)^2\le 1/\delta'^2.
\]
This shows that, if $C\ge1/\delta'^2$, then $s_n(\lambda)\le 2$
for all $n$, and consequently $s(\lambda)\le 2$.

To summarize, we have shown that, if $\dim_H(A)=0$, then $\dim_H(A_\lambda)=0$ for all $\lambda\in D(0,\rho)$ (combine \eqref{E:s(0)} and \eqref{E:s(lambda)}), and, if
$\dim_H(A)>0$, then $u:=1/s$ is an inf-harmonic function
on $D(0,\rho)$ such that
\[
u(0)=1/\dim_H(A)
\quad\text{and}\quad
1/2\le u(\lambda)\le1/\dim_H(A_\lambda)\quad(\lambda\in D(0,\rho)).
\]

The proof of the lemma is nearly complete,
except that $u$ is defined only on $D(0,\rho)$,
not on $\DD$. We fix this in exactly the same way as
at the end of the proof of Lemma~\ref{L:heartM}.
\end{proof}

\begin{remark}
Part (i) of Lemma~\ref{L:heartH} could also have been proved
using the well-known fact that the quasiconformal image of a set of
Hausdorff dimension zero also has Hausdorff dimension zero.
\end{remark}

\begin{proof}[Proof of Theorem~\ref{T:Hausdorff}]
We claim that, for each $\zeta\in\DD$,
if $\dim_H(A_\zeta)=0$,
then $\dim_H(A_\lambda)=0$ for all $\lambda\in\DD$,
and, if
$\dim_H(A_{\zeta})>0$,
then there exists an inf-harmonic function
$u_\zeta$ on $\DD$ such that
\[
u_\zeta(\zeta)=1/\dim_H(A_\zeta)
\quad\text{and}\quad
1/2\le u_\zeta(\lambda)\le 1/\dim_H(A_\lambda)
\quad(\lambda\in\DD).
\]
The special case $\zeta=0$ has been proved in
Lemma~\ref{L:heartH}, and the general case is deduced
from this just as in the proof of Theorem~\ref{T:Minkowski}
at the end of~\S\ref{S:Minkowski}.

Thus, either $\dim_H(A_\lambda)=0$ for all $\lambda\in\DD$, or
$\dim_H(A_\lambda)>0$ for all $\lambda\in\DD$. In the latter case,
we have
\[
\frac{1}{\dim_H(A_\lambda)}-\frac{1}{2}
=\sup_{\zeta\in\DD}(u_\zeta(\lambda)-1/2)
\quad(\lambda\in\DD),
\]
where the right-hand side is the supremum of a family
of functions that are inf-harmonic on $\DD$.
\end{proof}

%%%%%%%%%%%%%%%%%%%%%%%%%%%%%%%

\section{Proof of Theorem~\ref{T:converse}}\label{S:converse}

The essential idea of the proof is contained in the following lemma,
which is based on a construction in Astala's paper \cite{As94}.

\begin{lemma}\label{L:Astala}
Let $h:\DD\to(0,\infty)$ be a positive harmonic function, and let $n\ge10$.
Then there exists a holomorphic motion $\lambda\mapsto E_\lambda$ such that
$E_\lambda$ is a compact subset of $\DD$ for all $\lambda\in\DD$ and
\[
\frac{1}{\dim_H(E_\lambda)}
=\frac{1}{\dim_P(E_\lambda)}
=h(\lambda)+\frac{1}{2}+\frac{\log 2}{2\log n} \quad(\lambda\in\DD).
\]
\end{lemma}

\begin{proof}
As $h$ is a positive harmonic function on $\DD$,
there exists a holomorphic function $a:\DD\to\DD\setminus\{0\}$ such that
\[
\log|a(\lambda)|=-h(\lambda)\log n \quad (\lambda\in\DD).
\]

Let $\overline{D}(w_1,r),\dots,\overline{D}(w_n,r)$
be disjoint closed disks inside $\DD$,
where $r=1/\sqrt{2n}$.
Such disks may be found if $n\ge10$.
For $j=1,\dots,n$ and $\lambda\in\DD$, define
\[
\gamma_{j,\lambda}(z):=ra(\lambda)z+w_j \quad(z\in\CC).
\]

Note that $\gamma_{j,\lambda}(\DD)\subset D(w_j,r)$
for each $j=1,\dots,n$ and each $\lambda\in\DD$.
Thus, for each $\lambda\in\DD$,
the family $\{\gamma_{j,\lambda}: j=1,\dots,n\}$ generates
an iterated function system satisfying the open set condition.
If we denote by $E_\lambda$ its limit set,
then  $\lambda\mapsto E_\lambda$ is a
compact-valued holomorphic motion
(see e.g.\ \cite[Theorem~4]{BR06}) such that
$E_\lambda\subset \cup_{j=1}^n\overline{D}(w_j,r)\subset\DD$ for all $\lambda\in\DD$.
Moreover, by a special case of the Hutchinson--Moran formula Theorem~\ref{T:simdim},
the Hausdorff and packing dimensions of $E_\lambda$
are given by $\dim_HE_\lambda=\dim_P E_\lambda=s(\lambda)$,
where $s(\lambda)$ is the solution of the equation
\[
n (r|a(\lambda)|)^{s(\lambda)}=1.
\]
Solving this equation, we obtain
\[
\frac{1}{s(\lambda)}
=-\frac{\log(r|a(\lambda)|)}{\log n}
=\frac{\log(\sqrt{2n})+h(\lambda)\log n}{\log n}
=\frac{\log 2}{2\log n}+\frac{1}{2}+h(\lambda).
\]
This completes the proof.
\end{proof}

\begin{lemma}\label{L:infhn}
Let $D$ be a domain and let $u$ be an inf-harmonic function on~$D$.
Then there exists a sequence $(h_n)_{n\ge1}$
of positive harmonic functions on $D$ such that,
for every $m\ge1$, we have $u=\inf_{n\ge m}h_n$ on $D$.
\end{lemma}

\begin{proof}
Let $S$ be a countable dense subset of $D$,
and let $(\lambda_n)_{n\ge1}$ be a sequence in $S$
that visits every point of $S$ infinitely often.
Since $u$ is inf-harmonic on~$D$, for each $n\ge1$
there exists a positive harmonic function $h_n$ on $D$ such that
$h_n\ge u$ and $h_n(\lambda_n)<u(\lambda_n)+1/n$.
Then, for each $m\ge1$, we have $u=\inf_{n\ge m}h_n$ on $S$.
Since $S$ is dense in $D$ and inf-harmonic functions are automatically continuous,
it follows that $u=\inf_{n\ge m}h_n$ on $D$.
\end{proof}

\begin{proof}[Proof of Theorem~\ref{T:converse}]
Set $u(\lambda):=1/d(\lambda)-1/2$.
Since $1/d$ is inf-harmonic and $1/d\ge1/2$,
it follows that $u$ is inf-harmonic as well.
By Lemma~\ref{L:infhn},
there exists a sequence $(h_n)_{n\ge1}$
of positive harmonic functions on $\DD$ such that
$u=\inf_{n\ge m}h_n$ for every $m\ge1$.

By Lemma~\ref{L:Astala}, for each $n\ge10$,
there exists a compact-valued holomorphic motion
$\lambda\mapsto E^{(n)}_\lambda$ in $\DD$ such that
\[
\frac{1}{\dim_H(E^{(n)}_\lambda)}
=\frac{1}{\dim_P(E^{(n)}_\lambda)}
=h_n(\lambda)+\frac{1}{2}+\frac{\log 2}{2\log n}
\quad(\lambda\in\DD).
\]
Fix a sequence of disjoint closed disks $\overline{D}(\zeta_n,s_n)$ in $\CC$
such that $\zeta_n\to0$ and $s_n\to0$, and define
\[
A_\lambda:=\bigcup_{n\ge 10}(s_nE^{(n)}_\lambda+\zeta_n)\cup\{0\}
\quad(\lambda\in\DD).
\]
Then $\lambda\mapsto A_\lambda$ is a union of
holomorphic motions taking place in disjoint disks,
so it is itself a holomorphic motion.
Moreover $A_\lambda$ is a compact set for each $\lambda\in\DD$.
Finally, since both Hausdorff dimension and packing dimension
are countably stable,
and these dimensions are unchanged under similarities, we have
\begin{align*}
\frac{1}{\dim_H(A_\lambda)}
=\frac{1}{\dim_P(A_\lambda)}
&=\inf_{n\ge10}\biggl(\frac{1}{\dim_P(E^{(n)}_\lambda)}\biggr)\\
&=\inf_{n\ge10}\Bigl(h_n(\lambda)+\frac{1}{2}+\frac{\log 2}{2\log n} \Bigr)\\
&=u(\lambda)+\frac{1}{2}=\frac{1}{d(\lambda)}.
\end{align*}
In other words, $\dim_H(A_\lambda)=\dim_P(A_\lambda)=d(\lambda)$ for all $\lambda\in\DD$.
This completes the proof.
\end{proof}

%%%%%%%%%%%%%%%%%%%%%%%%%%%%%%%%%%

\section{Proof of Theorem~\ref{T:Area}}\label{S:Area}

In this section, we prove Theorem \ref{T:Area} on the variation of the area of a set moving under a holomorphic motion. The proof of part (i) follows closely the ideas of \cite{EH95}, as elaborated in \cite[\S13.1]{AIM09}. We first need the following lemmas.

\begin{lemma}\label{L:implicit2}
Let $(\Omega,\nu)$ be a  measure space and  let $a:\Omega\to(0,\infty)$ be a measurable function
such that $\int_\Omega a\,d\nu<\infty$.
Then, for every measurable function $p:\Omega\to(0,\infty)$
such that $\int_\Omega p\,d\nu=1$, we have
\[
\log\Bigl(\int_\Omega a\,d\nu\Bigr)
\ge\int_\Omega p\log\Bigl(\frac{a}{p}\Bigr)\,d\nu,
\]
with equality if $p=a/(\int_\Omega a\,d\nu)$.
\end{lemma}

\begin{proof}
The inequality follows from Jensen's inequality applied to the concave function $\log x$ and the probability space $(\Omega,\,p\,d\nu)$. The case of equality is obvious.
\end{proof}

\begin{lemma}\label{L:harmonic}
Let $D$ be a plane domain, let  $(\Omega,\nu)$ be a finite measure space,
and let $h:D\times\Omega\to\RR$ be a measurable function such that:
\begin{itemize}
\item $\lambda\mapsto h(\lambda,\omega)$ is harmonic on $D$, for each $\omega\in\Omega$;
\item $\sup_{K\times\Omega}|h(\lambda,\omega)|<\infty$ for each compact $K\subset D$.
\end{itemize}
Then the function $H(\lambda):=\int_\Omega h(\lambda,\omega)\,d\nu(\omega)$ is harmonic on $D$.
\end{lemma}

\begin{proof}
The function $H$ is continuous on $D$, by the dominated convergence theorem.
Also it satisfies the mean-value property on $D$, by the harmonicity of $h(\cdot,\omega)$ and Fubini's theorem.
Therefore $H$ is harmonic on $D$.
\end{proof}

\begin{lemma}\label{L:qcbounds}
Let $k\in(0,1)$ and let $R>0$.
Let $g,g_n:\CC\to\CC$ be $k$-quasiconformal homeomorphisms such that
\begin{itemize}
\item $\mu_{g_n} \to \mu_g$ a.e. on $\CC$,
\item $\supp \mu_{g_n}\subset D(0,R)$ for each $n \geq 1$,
\item $g_n(z)=z+o(1)=g(z)$ as $|z|\to\infty$ for each $n \geq 1$.
\end{itemize}
Then
\[
\|\partial_z g_n-\partial_z g\|_{L^2(\CC)}\to0
\quad\text{and}\quad \|\partial_{\overline{z}}g_n-\partial_{\overline{z}}g\|_{L^2(\CC)}\to0.
\]
\end{lemma}

\begin{proof}
The second limit holds by
\cite[Lemma~5.3.1]{AIM09}.
The first limit is an automatic consequence, since
$\|\partial_z g_n-\partial_z g\|_{L^2(\CC)}=\|\partial_{\overline{z}}g_n-\partial_{\overline{z}}g\|_{L^2(\CC)}$.
This is because the Beurling transform, which takes $\partial_{\overline{z}}f$ to $\partial_z f$, is a unitary operator on $L^2(\CC)$
(see the discussion on \cite[p.95]{AIM09}).
\end{proof}

\begin{proof}[Proof of Theorem~\ref{T:Area}]
Let $f:\DD\times\CC\to\CC$ be a holomorphic motion.
Suppose that there exists a compact subset $\Delta$ of $\CC$  such that, for each $\lambda\in\DD$,
the map $f_\lambda$ is conformal on $\CC\setminus\Delta$
and $f_\lambda(z)=z+O(1)$ near $\infty$.
Let $A$ be a Borel subset of $\Delta$ such that $|A|>0$. We begin with some preliminary remarks.

The first remark is that, in the normalization $f_\lambda(z)=z+O(1)$ near $\infty$,
we may as well suppose that in fact $f_\lambda(z)=z+o(1)$ near $\infty$.
Indeed, it suffices to consider the translated holomorphic motion $f(\lambda,z)-a_0(\lambda)$,
where $a_0(\lambda)$ is the constant coefficient in the Laurent expansion of $f_\lambda(z)$ near infinity.
Note that $a_0(\lambda)$ is holomorphic in $\DD$, as can be seen  from the formula
\[
a_0(\lambda)=\frac{1}{2\pi i}\int_{|z|=R}\frac{f_\lambda(z)}{z}\,dz,
\]
valid for all $R$ large enough so that $\Delta\subset D(0,R)$.

Next, we claim that there is a simple \emph{a priori} bound on $|A_\lambda|$, namely
\begin{equation}\label{E:aprioribound}
|A_\lambda|\le \pi c(\Delta)^2 \quad(\lambda\in \DD).
\end{equation}
Here $c(\Delta)$ is the logarithmic capacity of $\Delta$, see e.g. \cite[Chapter 5]{Ra95} for the definition. Indeed, since $f_\lambda$ is a conformal homeomorphism of $\CC\setminus\Delta$ onto $\CC\setminus f_\lambda(\Delta)$ satisfying $f_\lambda(z)=z+o(1)$ at infinity, the sets $\Delta$ and $f_\lambda(\Delta)$ have the same logarithmic capacity:
\[
c(f_\lambda(\Delta))=c(\Delta) \quad(\lambda\in\DD),
\]
by \cite[Theorem 5.2.3]{Ra95}. 
From the isoperimetric inequality for logarithmic capacity (\cite[Theorem 5.3.5]{Ra95}) we have
$|f_\lambda(\Delta)|\le \pi c(f_\lambda(\Delta))^2$,  and it follows that
\[
|A_\lambda|=|f_\lambda(A)|\le |f_\lambda(\Delta)|\le \pi c(f_\lambda(\Delta))^2=\pi c(\Delta)^2,
\]
as claimed.

We now turn to the proof of part~(i) of the theorem. Suppose first that
$A$ is compact and that there exists an open neighbourhood $U$ of $A$
such that $\mu_{f_\lambda}\equiv0$ on $U$ for all $\lambda\in\DD$.
Then each $f_\lambda$ is a conformal mapping on $U$, so $f_\lambda'(z)\ne0$ for all $z\in U$.
By the standard Jacobian formula for area, we have
\[
|A_\lambda|=|f_\lambda(A)|=\int_A |f_\lambda'(z)|^2\,dm(z),
\]
where $dm$ denotes area measure on $\CC$. Using Lemma~\ref{L:implicit2}, we can write
$\log|A_\lambda|$ as
\[
\log|A_\lambda|=\sup_p\Bigl\{\int_A p(z)\log\Bigl(\frac{|f_\lambda'(z)|^2}{p(z)}\Bigr)\,dm(z)\Bigr\},
\]
where the supremum is taken over all continuous functions $p:A\to(0,\infty)$ such that $\int_A p\,dm=1$.
By Lemma~\ref{L:harmonic},
each of the integrals is a harmonic function of $\lambda\in \DD$. Therefore $\log(C/|A_\lambda|)$ is
an inf-harmonic function on $\DD$ for each $C\ge \sup_{\lambda\in\DD}|A_\lambda|$,
in particular for $C=\pi c(\Delta)^2$, by \eqref{E:aprioribound}.

Suppose now that $A$ is merely Borel, but still that $\mu_{f_\lambda}\equiv0$ on $U$ for all $\lambda\in\DD$.
We have
\[
\log \Bigl(\frac{\pi c(\Delta)^2}{|A_\lambda|}\Bigr)=\inf_{F}\log\Bigl(\frac{\pi c(\Delta)^2}{|F_\lambda|}\Bigr) \quad (\lambda\in\DD),
\]
where the infimum is taken over all compact subsets $F$ of $A$.
Each function on the right-hand side is inf-harmonic on $\DD$,
by what we have already proved. Therefore the left-hand side is inf-harmonic on $\DD$ as well.

Finally, suppose merely that $\mu_{f_\lambda}=0$ a.e.\ on $A$ for each $\lambda\in \DD$.
Let $U_n$ be a deceasing sequence of bounded open sets such that $|U_n\setminus A|\to0$.
By Theorem \ref{T:motions}, for each $n$ there exists a holomorphic motion $f_n:\DD\times\CC\to\CC$ such that,
for each $\lambda\in\DD$, we have
$\mu_{f_{n,\lambda}}=1_{\CC\setminus U_n}\mu_{f_\lambda}$ a.e.\ on $\CC$
and $f_{n,\lambda}(z)=z+o(1)$ near $\infty$.
By Lemma~\ref{L:qcbounds}, it follows that
\[
\|\partial_zf_{n,\lambda}-\partial_z f_\lambda\|_{L^2(\CC)}\to0
\quad\text{and}\quad
 \|\partial_{\overline{z}}f_{n,\lambda}-\partial_{\overline{z}} f_\lambda\|_{L^2(\CC)}\to0.
\]
Therefore
\[
\int_A |\partial_zf_\lambda|^2\,dm=\lim_{n\to\infty}\int_A |\partial_zf_{n,\lambda}|^2\,dm=\lim_{n\to\infty}|f_{n,\lambda}(A)|
\]
and
\[
\int_A |\partial_{\overline{z}}f_\lambda|^2\,dm=\lim_{n\to\infty}\int_A |\partial_{\overline{z}}f_{n,\lambda}|^2\,dm=0.
\]
Hence, using \cite[formula (2.24)]{AIM09}, we obtain
\[
|f_\lambda(A)|=\int_A(|\partial_z f_\lambda|^2-|\partial_{\overline{z}}f_\lambda|^2)\,dm
=\lim_{n\to\infty}|f_{n,\lambda}(A)|.
\]
Thus
\[
\log\Bigl(\frac{\pi c(\Delta)^2}{|A_\lambda|}\Bigr)=
\lim_{n\to\infty}\log\Bigl(\frac{\pi c(\Delta)^2}{|f_{n,\lambda}(A)|}\Bigr)
\quad(\lambda\in\DD).
\]
By what we have already proved, the right-hand sides are inf-harmonic functions of $\lambda$.
It follows from Proposition~\ref{P:closure} part~\ref{I:limit} that the left-hand side is inf-harmonic on $\DD$ as well.
This completes the proof of part~(i) of the theorem.

We now turn to the proof of part~(ii).
Set $R_0:=\sup_{z\in\Delta}|z|$ and, for $R>R_0$, set $\Delta_R:=\overline{D}(0,R)$
(so $c(\Delta_R)=R$). By hypothesis $\mu_{f_\lambda}=0$ a.e.\ on $\Delta_R\setminus A$.
So, applying what we have proved in part~(i) (with $\Delta$ replaced by $\Delta_R$ and
$A$ replaced by $\Delta_R\setminus A$),
we see that
\[
\lambda\mapsto\log\Bigl(\frac{\pi R^2}{|f_\lambda(\Delta_R\setminus A)|}\Bigr)
\]
is an inf-harmonic function on $\DD$.
Now, fix $\lambda\in\DD$. Then $|f_\lambda(\Delta_R\setminus A)|=|f_\lambda(\Delta_R)|-|A_\lambda|$,
and by the area theorem from univalent function theory,
\[
|f_\lambda(\Delta_R)|=\pi R^2-\pi\sum_{n\ge1}n|a_n(\lambda)|^2R^{-2n},
\]
where $f_\lambda(z)=z+\sum_{n\ge1}a_n(\lambda)z^{-n}$ is the Laurent expansion of $f_\lambda$ near infinity.
In particular, $|f_\lambda(\Delta_R)|=\pi R^2+O(R^{-2})$ as $R\to\infty$. Hence
\[
\log\Bigl(\frac{\pi R^2}{|f_\lambda(\Delta_R\setminus A)|}\Bigr)
=\log\Bigl(\frac{\pi R^2}{\pi R^2-|A_\lambda|+O(R^{-2})}\Bigr)=\frac{|A_\lambda|}{\pi R^2}+O(R^{-4})\quad(R\to\infty).
\]
It follows that
\[
|A_\lambda|=\lim_{R\to\infty}\pi R^2\log\Bigl(\frac{\pi R^2}{|f_\lambda(\Delta_R\setminus A)|}\Bigr)
\quad(\lambda\in\DD).
\]
By what we have shown earlier,
the right-hand sides are inf-harmonic functions of $\lambda$.
It follows from Proposition~\ref{P:closure} part~\ref{I:limit} that the left-hand side is inf-harmonic on $\DD$ as well.
This completes the proof of part~(ii) of the theorem.
\end{proof}

%%%%%%%%%%%%%%%%%%%%%%%%%%%%%%%%%%

\section{Applications to quasiconformal maps}\label{S:quasiconformal}

In this section we show how our results
lead to a unified approach to the four theorems on quasiconformal distortion of area and dimension
that were stated at the end of the introduction.

\subsection{Distortion of dimension by quasiconformal maps}

In this subsection we establish Theorem~\ref{T:qcdistortion},
to the effect that, if $F:\CC\to\CC$ is a $k$-quasiconformal homeomorphism and $\dim A>0$, then
\begin{equation}\label{E:qcfdistortion}
\frac{1}{K}\Bigl(\frac{1}{\dim A}-\frac{1}{2}\Bigr)
\le \Bigl(\frac{1}{\dim F(A)}-\frac{1}{2}\Bigr)\le
K\Bigl(\frac{1}{\dim A}-\frac{1}{2}\Bigr),
\end{equation}
where $K=(1+k)/(1-k)$. Here
$\dim$ denotes any one of $\dim_P,\dim_H$ or $\overline{\dim}_M$. In the case of $\overline{\dim}_M$, we also suppose that the set $A$ is bounded.

\begin{proof}[Proof of Theorem~\ref{T:qcdistortion}]
By Theorem~\ref{T:embedding},
there exists a holomorphic motion $f:\DD\times\CC\to\CC$
such that $f_{k}=F$.
For $\lambda\in\DD$, set $A_\lambda:=f_\lambda(A)$.
By Theorems~\ref{T:Minkowski}, \ref{T:packing} and  \ref{T:Hausdorff},
either $\lambda\mapsto(1/\dim(A_\lambda)-1/2)$
is an inf-harmonic function on $\DD$, or, at the very least, it is a
supremum of inf-harmonic functions.
Either way, it satisfies Harnack's inequality,
so, for all $\lambda\in\DD$, we have
\[
\frac{1-|\lambda|}{1+|\lambda|}\Bigl(\frac{1}{\dim(A_0)}-\frac{1}{2}\Bigr)
\le \Bigl(\frac{1}{\dim(A_\lambda)}-\frac{1}{2}\Bigr)\le
\frac{1+|\lambda|}{1-|\lambda|}\Bigl(\frac{1}{\dim(A_0)}-\frac{1}{2}\Bigr).
\]
In particular, taking $\lambda=k$, we obtain
\eqref{E:qcfdistortion}.
\end{proof}

\begin{remark}
One consequence of Theorem~\ref{T:qcdistortion}
is that, if $f:\DD\times A\to\CC$ is a holomorphic motion and $A_\lambda=f_\lambda(A)$,
then the map $\lambda\mapsto\dim(A_\lambda)$
is a continuous function. For the Minkowski and packing dimensions, this was also proved in Corollary~\ref{C:subharmonic}.
For all three notions of dimension, it can also be seen
more directly as follows.

As $\lambda\to\lambda_0\in\DD$,
the transition map
$f_\lambda\circ f_{\lambda_0}^{-1}$
is $k$-quasiconformal with $k$ tending to $0$,
hence also H\"older-continuous
with H\"older exponent tending to~$1$
(see \cite[Theorem~12.2.3 and Corollary~3.10.3]{AIM09}).
Thus $\dim(A_\lambda)=\dim(f_\lambda\circ f_{\lambda_0}^{-1})(A_{\lambda_0})\to\dim(A_{\lambda_0})$ as $\lambda\to\lambda_0$.
\end{remark}
%%%%%%%%%%%%%%%%

\subsection{Distortion of area by quasiconformal maps}

\begin{proof}[Proof of Theorem \ref{T:areadistortion}]
Let $F:\CC\to\CC$ be a $k$-quasiconformal homeomorphism which is conformal on $\CC\setminus\Delta$,
where $\Delta$ is a compact set of logarithmic capacity at most~$1$, and such that $F(z)=z+o(1)$ near $\infty$. Let $A$ be a Borel subset of $\Delta$.

Let $k:=(K-1)/(K+1)$. By Theorem \ref{T:motions}, there is a holomorphic motion $f:\DD\times\CC\to\CC$ with $f_k=F$ and $\mu_{f_\lambda}=(\lambda/k)\mu_F$ for each $\lambda \in \mathbb{D}$. We may also require that $f_\lambda(z)=z+o(1)$ near $\infty$.

Suppose first that $\mu_F = 0$ a.e. on $A$. By Theorem~\ref{T:Area}(i), the function $\lambda\mapsto \log(\pi/|A_\lambda|)$ is inf-harmonic on $\DD$.
In particular, it satisfies Harnack's inequality there:
\[
\log\Bigl(\frac{\pi}{|A_\lambda|}\Bigr)\ge \frac{1-|\lambda|}{1+|\lambda|}\log\Bigl(\frac{\pi}{|A|}\Bigr)
\quad(\lambda\in\DD).
\]
Setting $\lambda=k$, we obtain
\[
\log\Bigl(\frac{\pi}{|F(A)|}\Bigr)\ge \frac{1}{K}\log\Bigl(\frac{\pi}{|A|}\Bigr).
\]
This proves (i).

Suppose instead that $\mu_F=0$ a.e. on $\CC \setminus A$. By Theorem~\ref{T:Area}(ii), the function $\lambda\mapsto|A_\lambda|$ is inf-harmonic on $\DD$.
In particular, it satisfies Harnack's inequality there:
\[
|A_\lambda|\le \frac{1+|\lambda|}{1-|\lambda|}|A|
\quad(\lambda\in\DD).
\]
Setting $\lambda=k$, we obtain
\[
|F(A)|\le K|A|.
\]
This proves (ii).

Finally, the general case (iii) is deduced from (i) and (ii) via a standard factorization process, see e.g. \cite[Theorem~13.1.4]{AIM09}.
\end{proof}

\begin{remark}
As mentioned in \S \ref{S:Area}, our proof of part~(i) of Theorem~\ref{T:areadistortion} is quite similar to the original proof of Eremenko and Hamilton \cite{EH95}, as presented in \cite[\S13.1]{AIM09}. On the other hand, our proof of part~(ii) is completely different from (and rather simpler than) the methods used in \cite{EH95} and \cite{AIM09}.
\end{remark}

\subsection{Symmetric holomorphic motions and inf-sym-harmonic functions}

In preparation for the proofs of Theorems~\ref{T:Smirnov} and
\ref{T:PrauseSmirnov},
we study what can be said about the function $1/\dim(A_\lambda)$ when $A$ is a subset of $\RR$ and $f:\mathbb{D} \times \CC \to \CC$ is a holomorphic motion that is symmetric in the sense defined below.

\begin{definition}\label{D:symmetricmotion}
We say that a holomorphic motion $f:\mathbb{D} \times \CC \to \CC$ is \textit{symmetric} if
\[
f_\lambda(z) = \overline{f_{\overline{\lambda}}(\overline{z})} \quad (\lambda \in \mathbb{D}, \,z \in \CC).
\]
\end{definition}

\begin{definition}
We say that a harmonic function $h:\mathbb{D} \to \RR$ is \textit{symmetric} if $h(\overline{\lambda})=h(\lambda)$ for all $\lambda \in \mathbb{D}$. A function $u:\mathbb{D} \to [0,\infty)$ is \textit{inf-sym-harmonic} if there is a family $\mathcal{H}$ of symmetric harmonic functions on $\mathbb{D}$ such that
\[
u(\lambda) = \inf_{h \in \mathcal{H}} h(\lambda) \qquad (\lambda \in \mathbb{D}).
\]
\end{definition}

We  now state  symmetric versions of Lemmas~\ref{L:heartM}
and \ref{L:heartH}.

\begin{lemma}\label{L:heartMsymmetric}
Let $f: \mathbb{D} \times \CC \to \CC$ be a symmetric holomorphic motion and let
 $A$ be a bounded subset of $\RR$ with $\overline{\dim}_M(A)>0$. Set $A_\lambda:=f_\lambda(A)$.
 Then there exists an inf-sym-harmonic function $u$ on $\mathbb{D}$ such that
\[
u(0)=1/\overline{\dim}_M(A)
\qquad\text{and}\qquad
u(\lambda)\ge 1/\overline{\dim}_M(A_\lambda)\quad(\lambda \in \mathbb{D}).
\]
\end{lemma}

\begin{lemma}\label{L:heartHsymmetric}
Let $f: \mathbb{D} \times \CC \to \CC$ be a symmetric holomorphic motion and let
 $A$ be a subset of $\RR$ with $\dim_H(A)>0$. Set $A_\lambda:=f_\lambda(A)$.
 Then there exists an inf-sym-harmonic function $u$ on $\mathbb{D}$ such that\[
u(0)=1/\dim_H(A)
\quad\text{and}\quad
1/2\le u(\lambda)\le1/\dim_H(A_\lambda)\quad(\lambda\in\DD).
\]
\end{lemma}

\begin{proof}
The proofs follow closely those of Lemmas~\ref{L:heartM} and \ref{L:heartH}, with the following differences:
\begin{itemize}
\item If $S\subset\RR$, then the function
$\log(M/\diam f_\lambda(S))$
defined in  Lemma \ref{L:diameter}
is inf-sym-harmonic on $D(0,\rho)$.
This can be seen directly from the formula
\[
\log\Bigl(\frac{M}{\diam f_\lambda(S)}\Bigr)=
\inf\Bigl\{\log\Bigl(\frac{M}{|f_\lambda(z)-f_\lambda(w)|}\Bigr):z,w\in S,\, z\ne w\Bigr\},
\]
using the symmetry relation $f_\lambda(z) = \overline{f_{\overline{\lambda}}(\overline{z})}$.
\item Consequently, if we replace the occurrences of $f_\lambda(D)$ and $f_\lambda(Q)$ in the proofs of  Lemmas~\ref{L:heartM} and \ref{L:heartH}
 by $f_\lambda(D\cap\RR)$
and $f_\lambda(Q\cap\RR)$ respectively,
then all the functions that were previously inf-harmonic
are now inf-sym-harmonic. Intersecting with $\RR$ leads to
no loss of information about $A$, since $A\subset\RR$.
\item When applying the implicit function theorem or its corollary
(Theorem~\ref{T:implicit} and Corollary~\ref{C:implicit}),
it is now assumed that the functions $\log(1/a_j)$ are inf-sym-harmonic,
and the conclusion is now that $1/s$ is inf-sym-harmonic
(or $s\equiv0$).
This follows by applying Lemma~\ref{L:implicit},
taking $\cU$ to be the inf-cone of inf-sym-harmonic functions.\qedhere
\end{itemize}
\end{proof}

%%%%%%%%%%%%%%%%%

\subsection{Dimension of quasicircles}

In this subsection, we establish Theorem~\ref{T:Smirnov}. More precisely, we use Lemma~\ref{L:heartHsymmetric} to show that the Hausdorff dimension of a $k$-quasicircle is at most $1+k^2$.

\begin{definition}\label{D:quasicircle}
Let $k\in[0,1)$.
A curve $\Gamma$ in $\CC$ is a $k$-\textit{quasicircle} if $\Gamma=g(\RR)$, where $g:\CC\to\CC$ is a normalized $k$-quasiconformal homeomorphism. By normalized, we mean simply that $g$ fixes $0$ and $1$.
\end{definition}

Quasicircles have been studied extensively over the years because of the desirable function-theoretic properties of the domains that they bound, see e.g.\ \cite{Ge12}. In particular, the problem of finding upper bounds for the Hausdorff dimension of a $k$-quasicircle in terms of $k$ has attracted much interest.
Theorem~\ref{T:qcdistortion} implies that if $\Gamma$ is a $k$-quasicircle, then
\[
\dim_H(\Gamma) \le 1+k.
\]
Motivated by examples of Becker and Pommerenke \cite{BP87}, Astala asked in \cite{As94} whether the upper bound can be replaced by $1+k^2$. This was answered in the affirmative by Smirnov in \cite{Sm10}. As we will now see, Astala's question can also be answered using inf-harmonic functions.

We first need a result on symmetrization of Beltrami coefficients due to Smirnov \cite[Theorem 4]{Sm10}. See also \cite[\S13.3.1]{AIM09}.

\begin{lemma}\label{L:SymBeltrami}
The function $g$ in Definition~\ref{D:quasicircle} may be chosen so
that, in addition, its Beltrami coefficient satisfies the antisymmetry relation
\begin{equation}\label{E:antisym}
\overline{\mu_g(\overline{z})} = -\mu_g(z) \quad  \text{a.e.\ in $\CC$}.
\end{equation}
\end{lemma}

We will also need the following Harnack-type inequality for inf-sym-har\-monic functions, reminiscent of \cite[Lemma 7]{Sm10}. See also \cite[Lemma~13.3.8]{AIM09}.

\begin{lemma}\label{L:HarnackSym}
Let $v:\mathbb{D} \to [0,\infty)$ be an inf-sym-harmonic function. Then
\[
\frac{1-y^2}{1+y^2}v(0) \le v(iy) \le \frac{1+y^2}{1-y^2} v(0) \qquad (y \in (-1,1)).
\]
\end{lemma}

\begin{proof}
Write
\[
v(\lambda) = \inf_{h \in \mathcal{H}} h(\lambda) \qquad (\lambda \in \mathbb{D}),
\]
where each $h \in \mathcal{H}$ is a positive and symmetric harmonic function on $\mathbb{D}$. Fix $h \in \mathcal{H}$, and set $k(\lambda):=(h(\lambda)+h(-\lambda))/2$. Clearly $k$ is an even positive harmonic function on $\mathbb{D}$. Thus it can be written as $k(\lambda)=l(\lambda^2)$, where $l$ is a positive harmonic function on $\mathbb{D}$. Applying the standard Harnack inequality to $l$, we get
\[
\frac{1-|\lambda|^2}{1+|\lambda|^2} k(0) \le k(\lambda) \le \frac{1+|\lambda|^2}{1-|\lambda|^2} k(0) \qquad (\lambda \in \mathbb{D}).
\]
As $h$ is symmetric, we have
\[
h(iy) = \frac{h(iy)+h(-iy)}{2} = k(iy) \qquad (y \in (-1,1)).
\]
Hence
\[
\frac{1-y^2}{1+y^2} h(0) \le h(iy) \le \frac{1+y^2}{1-y^2} h(0) \qquad (\lambda \in \mathbb{D}).
\]
Taking the infimum over all $h \in \mathcal{H}$ gives the result.
\end{proof}

We can now prove the main result of this subsection.

\begin{proof}[Proof of Theorem~\ref{T:Smirnov}]
Let $\Gamma$ be a $k$-quasicircle. By Lemma~\ref{L:SymBeltrami}, we can write $\Gamma=g(\RR)$ for  some normalized $k$-quasiconformal mapping $g: \CC \to \CC$ whose Beltrami coefficient $\mu_g$ satisfies the antisymmetry relation (\ref{E:antisym}).
For $\lambda \in \mathbb{D}$, define a Beltrami coefficient $\mu_\lambda$ by
\[
\mu_\lambda:= \frac{\lambda}{ik} \mu_g,
\]
and denote by $f_\lambda : \CC \to \CC$ the unique normalized quasiconformal mapping whose Beltrami coefficient is $\mu_\lambda$, as given by Theorem \ref{T:MRMT}. Note that $f_0$ is the identity and $f_{ik}=g$. It follows from Theorem \ref{T:motions} that the maps $f_\lambda$ define a holomorphic motion of $\CC$. Moreover, we have
\[
\overline{\mu_{\overline{\lambda}}(\overline{z})} = \frac{\lambda}{-ik} \overline{\mu_g(\overline{z})} = \mu_\lambda(z)
\quad \text{a.e.\ in $\CC$}.
\]
It easily follows that the maps $f_\lambda$ inherit the same symmetry:
\[
f_\lambda(z)=\overline{f_{\overline{\lambda}}(\overline{z})} \qquad (\lambda \in \mathbb{D}, z \in \CC),
\]
see e.g.\ \cite[Section 13.3.1]{AIM09}. In other words, the holomorphic motion $f$ is symmetric in the sense of Definition \ref{D:symmetricmotion}.

Now, let $A:=\RR$. By Lemma~\ref{L:heartHsymmetric}, there is an inf-sym-harmonic function $u$ on $\mathbb{D}$ such that
\[
u(0)=1/\dim_H(A)=1
\quad\text{and}\quad
1/2\le u(\lambda)\le1/\dim_H(A_\lambda)\quad(\lambda\in\DD).
\]
In particular, the function $v:=u-1/2$ is also inf-sym-harmonic, and Lemma \ref{L:HarnackSym} yields
\[
v(ik) \ge \frac{1-k^2}{1+k^2}v(0) = \frac{1}{2} \frac{1-k^2}{1+k^2}.
\]
But also
\[
v(ik) \le \frac{1}{\dim_H(f_{ik}(A))} - \frac{1}{2} = \frac{1}{\dim_H(\Gamma)} - \frac{1}{2},
\]
and hence we obtain
\[
\dim_H(\Gamma) \le 1 + k^2,
\]
as required.
\end{proof}

\begin{remark}
In fact, the upper bound in Theorem~\ref{T:Smirnov} is not sharp, as recently proved by Oleg Ivrii \cite{Iv15}.
\end{remark}

%%%%%%%%%%%%%%%%%%%%%%

\subsection{Quasisymmetric distortion spectrum}
In this subsection, we prove Theorem~\ref{T:PrauseSmirnov}. More precisely, we use Lemma \ref{L:heartMsymmetric} to estimate the Minkowski and packing dimensions of the image of a subset of the real line under a quasisymmetric map.

\begin{definition}
Let $k \in [0,1)$. A homeomorphism $g:\RR \to \RR$ is called $k$-\textit{quasisymmetric} if it extends to a normalized $k$-quasiconformal map $g: \CC \to \CC$ such that $g(z)=\overline{g(\overline{z})}$ for all $z \in \CC$.
\end{definition}

For the proof of Theorem~\ref{T:PrauseSmirnov}, we need the following Schwarz--Pick type inequality, see \cite[Lemma 2.2]{PS11}.

\begin{lemma}\label{L:SchwarzPick}
Let $\phi:\mathbb{D} \to \mathbb{D}$ be a holomorphic function. Suppose that $\phi(\lambda)=\overline{\phi(\overline{\lambda})}$ for all $\lambda \in \mathbb{D}$ and that $\phi(\lambda) \ge 0$ for all $\lambda \in (-1,1)$. Then
\[
\phi(k) \le \left( \frac{k+\sqrt{\phi(0)}}{1+k\sqrt{\phi(0)}} \right)^2 \qquad (0 \le k < 1).
\]
\end{lemma}

\begin{proof}[Proof of Theorem~\ref{T:PrauseSmirnov}]
It is enough to prove the result for the  Minkowski dimension.
The case of the packing dimension then follows easily
by applying Proposition~\ref{P:Minkowskipacking}.

Let $g:\RR \to \RR$ be a $k$-quasisymmetric map, and let $A \subset \RR$ be a bounded set with $\overline{\dim}_M(A)=\delta$, where $0<\delta \le 1$. It suffices to show that $\overline{\dim}_M(g(A)) \ge \Delta(\delta,k)$, since the upper bound follows from the lower bound, replacing $g$ by $g^{-1}$ and using the definition of $\Delta^*(\delta,k).$

Extend $g$ to a normalized $k$-quasiconformal mapping $g:\CC \to \CC$ such that $g(z)=\overline{g(\overline{z})}$ for all $z \in \CC$. The Beltrami coefficient $\mu_g$ satisfies
\[
\mu_g(z) = \overline{\mu_g(\overline{z})} \quad (z \in \CC).
\]
Therefore, by a similar construction to that in the proof of Theorem~\ref{T:Smirnov},  there is a symmetric holomorphic motion $f:\mathbb{D} \times \CC \to \CC$ with $f_k=g$.
By Lemma~\ref{L:heartMsymmetric},
there exists an inf-sym-harmonic function $u$ on $\mathbb{D}$ such that
\[
u(0)=1/\overline{\dim}_M(A)=1/\delta
\qquad\text{and}\qquad
u(\lambda)\ge 1/\overline{\dim}_M(A_\lambda)\quad(\lambda \in \mathbb{D}).
\]
The function $v:=u-1/2$ is also inf-sym-harmonic, and we can write
\[
v(\lambda) = \inf_{h \in \mathcal{H}} h(\lambda),
\]
where each $h \in \mathcal{H}$ is a positive, symmetric harmonic function on $\mathbb{D}$.

Fix $h \in \mathcal{H}$. Since $h$ is harmonic and $h(\overline{\lambda})=h(\lambda)$ for all $\lambda \in \mathbb{D}$, there is a holomorphic function $H$ on $\mathbb{D}$ with $\operatorname{Re} H = h$ and $\overline{H(\overline{\lambda})} = H(\lambda)$ for all $\lambda \in \mathbb{D}$. Then $H$ maps $\mathbb{D}$ into the right half-plane. Also, for $\lambda \in (-1,1)$, we have
\[
H(\lambda) = h(\lambda) \ge v(\lambda) \ge \frac{1}{\overline{\dim}_M(A_\lambda)}-\frac{1}{2} \ge \frac{1}{2},
\]
since $A_\lambda \subset \RR$ by the symmetry of the holomorphic motion. It follows that the function
\[
\phi:= \frac{2H-1}{2H+1}
\]
satisfies the assumptions of Lemma \ref{L:SchwarzPick}, and we get
\[
\frac{2h(k)-1}{2h(k)+1} = \phi(k) \le \left( \frac{k+l'}{1+kl'} \right)^2,
\]
where $l'=\sqrt{\phi(0)}$. Using the fact that the functions $x \mapsto (2x-1)/(2x+1)$ and $x \mapsto (k+x)/(1+kx)$ are increasing, we obtain, after taking the infimum over all $h \in \mathcal{H}$,
\[
\frac{2v(k)-1}{2v(k)+1} \le \left( \frac{k+l}{1+kl} \right)^2,
\]
where
\[
l = \left( \frac{2v(0)-1}{2v(0)+1} \right)^{1/2} = \left( \frac{2(1/\delta - 1/2)-1}{2(1/\delta-1/2)+1} \right)^{1/2} = \sqrt{1-\delta}.
\]
Note that
\[
\frac{2v(k)-1}{2v(k)+1}
= \frac{2u(k)-2}{2u(k)}
= 1-\frac{1}{u(k)} \ge 1- \overline{\dim}_M(g(A)).
\]
This gives the desired inequality, namely
\[
\overline{\dim}_M(g(A)) \ge 1-\left( \frac{k+l}{1+kl} \right)^2= \Delta(\delta,k).\qedhere
\]
\end{proof}

\section{An open problem}\label{S:conclusion}

As remarked in the introduction, Theorems~\ref{T:packing}
and \ref{T:converse} between them provide a complete characterization
of the variation of the packing dimension of a set moving under a holomorphic motion.
Such a characterization for the Hausdorff dimension is currently lacking, due to the fact that the conclusion in Theorem~\ref{T:Hausdorff} is weaker than that in
Theorem~\ref{T:packing}. This naturally raises the following question.

\begin{question}\label{Q:Hausdorff}
Let $A$ be a subset of $\CC$ such that $\dim_HA>0$,
and let $f:\DD\times A\to\CC$ be a holomorphic motion.
Set $A_\lambda:=f_\lambda(A)$.
Then must $\lambda\mapsto 1/\dim_H(A_\lambda)$
be an inf-harmonic function on $\DD$?
\end{question}

The same question was posed 30 years ago in \cite{Ra93}.
As far as we know, it is still an open problem.

It was shown in \cite{Ra93} that the answer to Question~\ref{Q:Hausdorff} is affirmative in the following
special case.
Let $(R_\lambda)_{\lambda\in\DD}$
be a holomorphic family of hyperbolic rational maps.
Then the holomorphic motion $\lambda\mapsto J(R_\lambda)$
defined by their Julia sets has the property that
$1/\dim_H J(R_\lambda)$ is an inf-harmonic function on~$\DD$.
The proof relies on an explicit formula for the Hausdorff dimension,
namely the Bowen--Ruelle--Manning formula.

Another special case was established by Baribeau and Roy \cite{BR06}.
They showed that,
if $L_\lambda$ is the limit set of an iterated function system
of contractive similarities
depending holomorphically on a parameter $\lambda\in \DD$,
then, subject to a technical condition,
the map $\lambda\mapsto L_\lambda$ is a holomorphic motion for which
$1/\dim_H(L_\lambda)$ is an inf-harmonic function on~$\DD$.
Their proof also relies on an explicit formula for the the Hausdorff dimension,
this time the Hutchinson--Moran formula,
Theorem~\ref{T:simdim}.

In fact, in both these special cases, it turns out that the Hausdorff dimension coincides with the packing dimension,
so both results are now consequences of Theorem~\ref{T:packing},
without any recourse to explicit formulas for the dimension.

Finally, we remark that an affirmative answer to
Question~\ref{Q:Hausdorff} would imply that
$\lambda\mapsto\dim_H(A_\lambda)$ is a subharmonic function
(in much the same way that Corollary~\ref{C:subharmonic}
was proved for the packing and Minkowski dimensions).
Even this apparently weaker statement is also still an open problem.
As an interesting test case,
we pose the following question.

\begin{question}\label{Q:sh}
Does the Hausdorff dimension of  a holomorphic motion
$\lambda\mapsto A_\lambda$ always satisfy the inequality
\[
\dim_H(A_0)\le \max_{|\lambda|=1/2}\dim_H(A_\lambda)?
\]
\end{question}

%%%%%%%%%%%%%%%%%%%%%%%%%%%%%%%%%%

\bibliographystyle{amsplain}
\bibliography{biblist.bib}

\providecommand{\bysame}{\leavevmode\hbox to3em{\hrulefill}\thinspace}
\providecommand{\MR}{\relax\ifhmode\unskip\space\fi MR }
\providecommand{\MRhref}[2]{%
  \href{http://www.ams.org/mathscinet-getitem?mr=#1}{#2}
}
\providecommand{\href}[2]{#2}
\begin{thebibliography}{10}

\bibitem{As94}
K.~Astala, \emph{Area distortion of quasiconformal mappings}, Acta Math.
  \textbf{173} (1994), no.~1, 37--60. \MR{1294669}

\bibitem{AIM09}
K.~Astala, T.~Iwaniec, and G.~Martin, \emph{Elliptic partial differential
  equations and quasiconformal mappings in the plane}, Princeton Mathematical
  Series, vol.~48, Princeton University Press, Princeton, NJ, 2009.
  \MR{2472875}

\bibitem{AZ95}
K.~Astala and M.~Zinsmeister, \emph{Holomorphic families of quasi-{F}uchsian
  groups}, Ergodic Theory Dynam. Systems \textbf{14} (1994), no.~2, 207--212.
  \MR{1279468}

\bibitem{BR06}
L.~Baribeau and M.~Roy, \emph{Analytic multifunctions, holomorphic motions and
  {H}ausdorff dimension in {IFS}s}, Monatsh. Math. \textbf{147} (2006), no.~3,
  199--217. \MR{2215564}

\bibitem{BP87}
J.~Becker and Ch. Pommerenke, \emph{On the {H}ausdorff dimension of
  quasicircles}, Ann. Acad. Sci. Fenn. Ser. A I Math. \textbf{12} (1987),
  no.~2, 329--333. \MR{951982}

\bibitem{BP17}
C.~J. Bishop and Y.~Peres, \emph{Fractals in probability and analysis},
  Cambridge Studies in Advanced Mathematics, vol. 162, Cambridge University
  Press, Cambridge, 2017. \MR{3616046}

\bibitem{CT23}
E.~K. Chrontsios~Garitsis and J.~T. Tyson, \emph{Quasiconformal distortion of
  the {A}ssouad spectrum and classification of polynomial spirals}, Bull.
  London Math. Soc. \textbf{55} (2023), no.~1, 282--307.

\bibitem{EH95}
A.~Er\"{e}menko and D.~H. Hamilton, \emph{On the area distortion by
  quasiconformal mappings}, Proc. Amer. Math. Soc. \textbf{123} (1995), no.~9,
  2793--2797. \MR{1283548}

\bibitem{Fa14}
K.~Falconer, \emph{Fractal geometry}, third ed., John Wiley \& Sons, Ltd.,
  Chichester, 2014, Mathematical foundations and applications. \MR{3236784}

\bibitem{Ge12}
F.~W. Gehring and K.~Hag, \emph{The ubiquitous quasidisk}, Mathematical Surveys
  and Monographs, vol. 184, American Mathematical Society, Providence, RI,
  2012, With contributions by Ole Jacob Broch. \MR{2933660}

\bibitem{GV73}
F.~W. Gehring and J.~V\"{a}is\"{a}l\"{a}, \emph{Hausdorff dimension and
  quasiconformal mappings}, J. London Math. Soc. (2) \textbf{6} (1973),
  504--512. \MR{324028}

\bibitem{Hu81}
J.~E. Hutchinson, \emph{Fractals and self-similarity}, Indiana Univ. Math. J.
  \textbf{30} (1981), no.~5, 713--747. \MR{625600}

\bibitem{Iv15}
O.~Ivrii, \emph{Quasicircles of dimension $1+k^2$ do not exist},
  arXiv:1511.07240, 2015.

\bibitem{Ka00}
R.~Kaufman, \emph{Sobolev spaces, dimension, and random series}, Proc. Amer.
  Math. Soc. \textbf{128} (2000), no.~2, 427--431. \MR{1670383}

\bibitem{MSS83}
R.~Ma\~{n}\'{e}, P.~Sad, and D.~Sullivan, \emph{On the dynamics of rational
  maps}, Ann. Sci. \'{E}cole Norm. Sup. (4) \textbf{16} (1983), no.~2,
  193--217. \MR{732343}

\bibitem{Mo46}
P.~A.~P. Moran, \emph{Additive functions of intervals and {H}ausdorff measure},
  Proc. Cambridge Philos. Soc. \textbf{42} (1946), 15--23. \MR{14397}

\bibitem{Pr07}
I.~Prause, \emph{A remark on quasiconformal dimension distortion on the line},
  Ann. Acad. Sci. Fenn. Math. \textbf{32} (2007), no.~2, 341--352. \MR{2337481}

\bibitem{PS11}
I.~Prause and S.~Smirnov, \emph{Quasisymmetric distortion spectrum}, Bull.
  Lond. Math. Soc. \textbf{43} (2011), no.~2, 267--277. \MR{2781207}

\bibitem{Ra93}
T.~Ransford, \emph{Variation of {H}ausdorff dimension of {J}ulia sets}, Ergodic
  Theory Dynam. Systems \textbf{13} (1993), no.~1, 167--174. \MR{1213086}

\bibitem{Ra95}
\bysame, \emph{Potential theory in the complex plane}, London Mathematical
  Society Student Texts, vol.~28, Cambridge University Press, Cambridge, 1995.
  \MR{1334766}

\bibitem{Ru82}
D.~Ruelle, \emph{Repellers for real analytic maps}, Ergodic Theory Dynam.
  Systems \textbf{2} (1982), no.~1, 99--107. \MR{684247}

\bibitem{Sl91}
Z.~Slodkowski, \emph{Holomorphic motions and polynomial hulls}, Proc. Amer.
  Math. Soc. \textbf{111} (1991), no.~2, 347--355. \MR{1037218}

\bibitem{Sm10}
S.~Smirnov, \emph{Dimension of quasicircles}, Acta Math. \textbf{205} (2010),
  no.~1, 189--197. \MR{2736155}

\bibitem{ST86}
D.~P. Sullivan and W.~P. Thurston, \emph{Extending holomorphic motions}, Acta
  Math. \textbf{157} (1986), no.~3-4, 243--257. \MR{857674}

\bibitem{Tr82}
C.~Tricot, Jr., \emph{Two definitions of fractional dimension}, Math. Proc.
  Cambridge Philos. Soc. \textbf{91} (1982), no.~1, 57--74. \MR{633256}

\end{thebibliography}

\end{document}